\documentclass[11pt,reqno]{amsart}
\usepackage{mathrsfs}
\usepackage{cite}

\def\subjclass#1{{\renewcommand{\thefootnote}{}%
\footnote{\emph{Mathematics Subject Classification (2020):} #1}}}

\usepackage{amsmath}
\usepackage{amssymb}
\usepackage{amsfonts}

\DeclareMathOperator{\divg}{div}

\date{\today}

\hoffset=- 2cm \voffset=- 0cm 
 \theoremstyle{plain}
\newtheorem{Thm}{Theorem}[section]
\newtheorem{Lem}{Lemma}[section]
\newtheorem{Assum}{Assumption}[section]
\newtheorem{Prop}{Proposition}[section]
\newtheorem{Rem}{Remark}[section]
\newtheorem{Cor}{Corollary}[section]

\usepackage{amssymb}
\usepackage{color}

\def\0{\mathbf 0}

\def\v{\vskip}

\errorcontextlines=0 \numberwithin{equation}{section}
\numberwithin{Thm}{section}
\allowdisplaybreaks

\begin{document}
\large

\title[Liouville type theorems]
{Some new Liouville type theorems for 3D steady tropical climate model}

\author{Yanyan Dong}

\address{Yanyan Dong: School of Mathematics and Statistics, Anqing Normal University, Anqing 246133, China}
\email{2981165080@qq.com}

\author{Yan Fang}

\address{Yan Fang: School of Mathematics and Statistics, Anqing Normal University, Anqing 246133, China}
\email{3198132434@qq.com}

\author{Zhibing Zhang}

\address{Zhibing Zhang: School of Mathematics and Statistics, Key Laboratory of Modeling, Simulation and Control of Complex Ecosystem in Dabie Mountains of Anhui Higher Education Institutes, Anqing Normal University, Anqing 246133, China}
\email{zhibingzhang29@126.com}%

\thanks{}

\keywords{Liouville type theorems, tropical climate model, power-law growth conditions, logarithmic improvement}

\subjclass{35B53, 35Q35, 35A02}

\begin{abstract}
In this paper, we establish two major classes of Liouville type results for the three-dimensional stationary tropical climate model.
The first class is obtained under the assumptions imposed on $u,v,\theta$ whereas the second one relies on the assumptions imposed on $u,v,\nabla\theta$.
Using the energy method and an iteration argument, we obtain Liouville type theorems under the condition that Lebesgue norms of the smooth solutions on the annulus satisfy some power-law growth conditions. As a consequence, we show that a smooth solution is trivial provided that it belongs to some Lebesgue spaces or satisfies some decay conditions at infinity. Furthermore, with the aid of a contradiction argument and by developing a systematic framework to handle the energy function associated with the non-trivial solutions, we obtain a logarithmic improvement of our Liouville type theorems. Our new framework is very effective for establishing logarithmic improvements of Liouville type theorems for coupled fluid equations.
\end{abstract}
\maketitle

\tableofcontents

\section{Introduction}
We consider the following stationary tropical climate model in $\mathbb{R}^3$:
\begin{equation}\label{equ1.1}
  \left\{
    \begin{array}{ll}
     -\Delta u+(u\cdot\nabla) u+\nabla \pi+\divg(v\otimes v)=0,  \\
   -\Delta v+(u\cdot\nabla) v+\nabla \theta+(v\cdot\nabla)u=0, \\
   -\Delta \theta+u\cdot\nabla\theta+\divg v=0,\\
     \divg u=0,
\end{array}
  \right.
\end{equation}
where $u= (u_1 , u_2 , u_3 )$ is the barotropic mode, $v= (v_1 , v_2 , v_3 )$ is the first baroclinic mode of vector velocity, $\theta$ is the temperature, and $\pi$ is the pressure.
When $v=0$ and $\theta$ is a constant, \eqref{equ1.1} is reduced to the Navier-Stokes equations.
The tropical climate model is derived by Frierson, Majda and Pauluis \cite{FMP04} for large-scale dynamics of precipitation fronts in the tropical atmosphere.
More relevant background on the tropical climate model can be found in \cite{Majda03,MB03} and the references therein.
Mathematically, \eqref{equ1.1} exhibits a distinctive structure arising from the coupling between the divergence-free vector field $u$ and the non-divergence-free vector field $v$. Such mixed coupling structure brings essential mathematical difficulties and analytical challenges.

In recent years, Liouville type theorems for the fluid equations have attracted increasing attention. There are many interesting Liouville type results for the Navier-Stokes equations.
In this paper, we only refer to several results relevant to ours. For instance, Galdi  \cite{Galdi} proved  that $u$ is a trivial solution under the assumption that $u \in L^{\frac{9}{2}}(\mathbb{R}^{3})$. Subsequently, Chae and Wolf \cite{CW16} made a logarithmic improvement to Galdi's result by assuming that $u$ satisfies
$$\int_{\mathbb{R}^{3}}|u|^{\frac{9}{2}}\left\{\log\left(2+\frac{1}{|u|}\right)\right\}^{-1}dx<+\infty.$$
Seregin \cite{Seregin18} handled the pressure term by using a Bogovskii map on the ball, and proved that $u$ is trivial if $u$ fulfils
$$\displaystyle\sup_{R>0}R^{\beta}\left(\frac{1}{|B_{R}|}\int_{B_{R}}|u|^pdx\right)^{\frac{1}{p}}<+\infty, ~~~~~~~~~~~~with~~~~~~~~~~~~\frac{3}{2}<p<3~~~~~~~~~~~~and~~~~~~~~~~~~\beta>\frac{6p-3}{8p-6}.$$
Recently, Cho et al. \cite{CNY24} got an improvement on the result of Seregin. They showed that $u$ is trivial if $u$ satisfies one of the following conditions
$$
\aligned
&(\mathrm{i})\liminf_{R\rightarrow+\infty}R^{-\left(\frac{2}{p}-\frac{1}{3}\right)}\|u\|_{L^p(B_{2R}\backslash B_R)}<+\infty,\;\frac{3}{2}<p<3,\\
&(\mathrm{ii})\liminf_{R\rightarrow+\infty}R^{-\frac{1}{3}}\|u\|_{L^3(B_{2R}\backslash B_R)}=0.
\endaligned
$$
Very recently, Cho-Yang \cite{CY25} got a logarithmic improvement of Liouville type results. Specifically, they proved the triviality of the smooth solution if
$$
\limsup_{R\rightarrow+\infty}\frac{\|u\|_{L^p(B_{2R}\backslash B_\frac{3R}{2})}}{R^{\frac{2}{p}-\frac{1}{3}}(\ln R)^{\frac{3}{p}-1}}<+\infty,\;\frac{3}{2}<p<3.
$$
For more Liouville type results for the Navier-Stokes equations, see \cite{BGWX25,CPZ20,CPZZ20,Chae14,Chae25,Chae26,CW19,DJL21,CoY26,KNSS09,KTW17,Seregin16,SW19} and the references therein.

However, there are relatively few Liouville type results for the tropical climate model so far. Ding and Wu \cite{FW21} investigated the tropical climate model and proved Liouville type theorems under one of the following assumptions
$$
\aligned
&(\mathrm{i})~~~~u\in L^p\left(\mathbb{R}^{3}\right) ~~~~~~~~~~~~\text{and}~~~~~~~~~~~~v,\theta\in L^p\left(\mathbb{R}^{3}\right)\cap L^3\left(\mathbb{R}^{3}\right) ~~~~~~~~~~~~\text{with}~~~~~~~~~~~~ 3\leq p\leq \frac{9}{2},\\
&(\mathrm{ii})~~~~u\in L^{3}\left(\mathbb{R}^{3}\right), v \in L^{2}\left(\mathbb{R}^3\right)~~~~~~~~~~~~\text{and}~~~~~~~~~~~~\nabla u,\nabla v,\nabla\theta \in L^{2}\left(\mathbb{R}^3\right).
\endaligned
$$
Later, Cho et al. \cite{CIY23} improved the result of Ding and Wu by removing the condition $\nabla u,\nabla v \in L^{2}(\mathbb{R}^3)$ from the above assumption $(\mathrm{ii})$.
Recently, Cho et al. \cite{CIY24} established an improved Liouville type theorem for \eqref{equ1.1} under one of the following conditions
\begin{subequations}
\begin{align}
&(\mathrm{i})\;p\in\left(\frac{3}{2},3\right),\;q,r\in[1,2),\;\liminf_{R\rightarrow+\infty}[G(u;p,R)+H(v;q,R)+H(\theta;r,R)]<+\infty, \label{1.2a}\\
&(\mathrm{ii})\liminf_{R\rightarrow+\infty}[G(u;3,R)+H(v;2,R)]=0,\;\limsup_{R\rightarrow+\infty}H(\theta;2,R)<+\infty,\label{1.2b}\\
&(\mathrm{iii})\liminf_{R\rightarrow+\infty}[G(u;3,R)+H(\theta;2,R)]=0,\;\limsup_{R\rightarrow+\infty}H(v;2,R)<+\infty,\label{1.2c}
\end{align}
\end{subequations}
where the functions $G$ and $H$ are defined respectively by
$$ G(f;p,R):=R^{-\left(\frac{2}{p}-\frac{1}{3}\right)}\|f\|_{L^p(B_{2R}\backslash B_{\frac{R}{2}})}
~~~~~~~~~~~~and~~~~~~~~~~~~
H(f;p,R):=R^{-\left(\frac{3}{2p}-\frac{1}{4}\right)}\|f\|_{L^p(B_{2R}\backslash B_{\frac{R}{2}})}.
$$

Our aim of this paper is to extend and improve the recent works of \cite{CIY23,CIY24,CNY24,CY25}.
By means of the energy method and an iteration argument, we establish Liouville type theorems under the condition that Lebesgue norms of the smooth solutions on the annulus satisfy some power-law growth conditions. As a consequence, we show that a smooth solution is trivial provided that it belongs to some Lebesgue spaces or satisfies some decay conditions at infinity.
It is not difficult to verify that Liouville type theorem under \eqref{1.2a} is a special case of Theorem \ref{main1} while Liouville type theorem under \eqref{1.2b} or \eqref{1.2c} is a special case of Theorem \ref{main1} and Remark \ref{Rem3.1}(iii). Moreover, in our setting, the growth conditions are more flexible and the ranges of the parameters are broader.

On the other hand, using a contradiction argument and developing a systematic framework to handle the energy function associated with the non-trivial solutions, we extend Cho-Yang's logarithmic improvement of Liouville type theorems for the Navier-Stokes equations to \eqref{equ1.1}. In \cite{CY25}, Cho-Yang proved their logarithmic improvement result via a case-by-case analysis based on
\begin{align*}
\liminf_{R\rightarrow+\infty}\frac{\|u\|_{L^3(B_{2R}\backslash B_\frac{3R}{2})}}{R^{\frac{2}{3}-\frac{1}{p}}(\ln R)^{\frac{3}{p}-1}}<+\infty\text{ or }=+\infty.
\end{align*}
To the best of our knowledge, their method cannot be applied to establish Liouville type results of the logarithmic version for the fluid equations coupled with the Navier-Stokes equations. Our new approach avoids the case-by-case discussion and is very effective for the coupled fluid equations.

In this paper, we establish two major classes of Liouville type results. The first class is obtained under the assumptions imposed on $u,v,\theta$ whereas the second one relies on the assumptions imposed on $u,v,\nabla\theta$. In order to show our Liouville type results, we need to introduce some notations. Let $A_R$ denote the annulus $B_{2R}\backslash \overline{B_\frac{3R}{2}}$.  For simplicity, we denote
\begin{align*}
&X_{p,\alpha}(R)=R^{-\alpha}\|u\|_{L^p\left(A_{R}\right)},\;X_{p,\alpha,\lambda}(R)=R^{-\alpha}(\ln R)^{-\lambda}\|u\|_{L^p\left(A_{R}\right)},\\
&Y_{q,\beta}(R)=R^{-\beta}\|v\|_{L^q\left(A_{R}\right)},\;Y_{q,\beta,\mu}(R)=R^{-\beta}(\ln R)^{-\mu}\|v\|_{L^q\left(A_{R}\right)},\\
&Z_{r,\gamma}(R)=R^{-\gamma}\|\theta\|_{L^r\left(A_{R}\right)},\;Z_{r,\gamma,\nu}(R)=R^{-\gamma}(\ln R)^{-\nu}\|\theta\|_{L^r\left(A_{R}\right)},\\
&Z'_{r,\gamma}(R)=R^{-\gamma}\|\nabla\theta\|_{L^r\left(A_{R}\right)},\;Z'_{r,\gamma,\nu}(R)=R^{-\gamma}(\ln R)^{-\nu}\|\nabla\theta\|_{L^r\left(A_{R}\right)}.
\end{align*}

\subsection{Liouville type theorems under some assumptions imposed on $u,v,\theta$}
Let $p'$ denote the conjugate exponent to $p$, i.e., $p'=\frac{p}{p-1}$.
Before stating our first class of Liouville type results, we introduce some basic assumptions on the parameters $p,q,r$ and $\alpha$, $\beta$, $\gamma$.
\begin{Assum}\label{a1.1}
Let $(p,q,r)\in\left(\frac{3}{2},\frac{9}{2}\right]\times[1,6]\times[1,6]$ and $\alpha\in\left[0,\alpha_p\right]$, $\beta\in\left[0,\frac{3}{q}-\frac{1}{2}\right]$, $\gamma\in\left[0,\frac{3}{r}-\frac{1}{2}\right]$, where
\begin{equation*}
\alpha_p=
\begin{cases}
\frac{2}{p}-\frac{1}{3}, & \text{ if  }\frac{3}{2}<p<3, \\
\frac{3}{p}-\frac{2}{3}, & \text{ if  }3\leq p\leq\frac{9}{2}.
\end{cases}
\end{equation*}
\end{Assum}

\begin{Assum}\label{a1.2}
\begin{itemize}
		\item[(i)] When $(p,q)\in\left(\frac{3}{2},\frac{9}{2}\right]\times[1,6]$, it holds that
$$
\alpha+\frac{(4p-6)\min\{q,2p'\}}{(6-\min\{q,2p'\})p}\beta\leq\frac{3}{p}+\frac{6}{\max\{q,2p'\}}-2.
$$
In addition, when $(p,q)\in\left(\frac{3}{2},3\right)\times[2p',6]$, equality can not be achieved in the above inequality.
        \item[(ii)] When $(p,r)\in\left(\frac{3}{2},\frac{9}{2}\right]\times[1,6]$, it holds that
$$
\alpha+\frac{(4p-6)\min\{r,2p'\}}{(6-\min\{r,2p'\})p}\gamma\leq\frac{3}{p}+\frac{6}{\max\{r,2p'\}}-2.
$$
In addition, when $(p,r)\in\left(\frac{3}{2},3\right)\times[2p',6]$, equality can not be achieved in the above inequality.
\item[(iii)] When $(q,r)\in [1,6]\times[1,6]$, it holds that
$$
\frac{2\min\{q,2\}}{6-\min\{q,2\}}\beta+\frac{2\min\{r,2\}}{6-\min\{r,2\}}\gamma\leq\frac{3}{\max\{q,2\}}+\frac{3}{\max\{r,2\}}-2.
$$
In addition, when $(q,r)\in[2,6]\times[2,6]$, equality can not be achieved in the above inequality.
\end{itemize}
\end{Assum}

\begin{Assum}\label{a1.3}
Let $\lambda\in\left[0,\lambda_p\right]$, $\mu,\nu\geq0$, where
\begin{equation*}
\lambda_p=
\begin{cases}
\frac{3}{p}-1, & \text{ if  }\frac{3}{2}<p<3, \\
0, & \text{ if  }3\leq p\leq\frac{9}{2}.
\end{cases}
\end{equation*}
\end{Assum}

\begin{Assum}\label{a1.4}
\begin{itemize}
		\item[(i)] When $q\in[1,2p')$, it holds that
$$\lambda+\frac{(4p-6)q}{(6-q)p}\mu\leq\frac{6p-3q(p-1)}{(6-q)p}.$$
        \item[(ii)] When $r\in[1,2p')$, it holds that
$$\lambda+\frac{(4p-6)r}{(6-r)p}\nu\leq\frac{6p-3r(p-1)}{(6-r)p}.$$
        \item[(iii)] When $(q,r)\in [1,6]\times[1,2)$, or $[1,2)\times[2,6]$, it holds that
$$
\frac{2\min\{q,2\}}{6-\min\{q,2\}}\mu+\frac{2\min\{r,2\}}{6-\min\{r,2\}}\nu\leq\frac{6-3\min\{q,2\}}{12-2\min\{q,2\}}+\frac{6-3\min\{r,2\}}{12-2\min\{r,2\}}.
$$
\end{itemize}		
\end{Assum}

\begin{Assum}\label{a1.5}
\begin{itemize}
		\item[(i)] When $\beta=\frac{3}{q}-\frac{1}{2}$, we require $0\leq\mu\leq\frac{1}{2}$.
        \item[(ii)] When $(p,q)\in\left(\frac{3}{2},\frac{9}{2}\right]\times[2p',6]$, we require $\alpha+2\beta<\frac{3}{p}+\frac{6}{q}-2$.
        \item[(iii)] When $(p,q)\in\left(\frac{3}{2},3\right]\times[1,2p')$, we require $(\alpha,\beta)\neq\left(\frac{3}{p}-1,\frac{3}{q}-\frac{1}{2}\right)$.
        \item[(iv)] When $1\leq q<2$, we require $(r,\beta,\gamma)\neq(2,\frac{3}{q}-\frac{1}{2},0)$.
        \item[(v)] When $1\leq r<2$, we require $(q,\beta,\gamma)\neq(2,0,\frac{3}{r}-\frac{1}{2})$.
\end{itemize}
\end{Assum}

We are now in a position to state our first Liouville type theorem under the assumption that the Lebesgue norms of smooth solutions on the annulus satisfy certain power-law growth conditions.
\begin{Thm}\label{main1}
Let $(u,\pi,v,\theta)$ be a smooth solution of \eqref{equ1.1}. Suppose the parameters satisfy {\rm Assumptions \ref{a1.1} and \ref{a1.2}}. Furthermore, assume that one of the following assumptions holds
\begin{align*}
&\mathrm{(A1)}\;p\in\left[3,\frac{9}{2}\right],\;\liminf\limits_{R\rightarrow+\infty}X_{p,\alpha}(R)=0,\;\limsup\limits_{R\rightarrow+\infty}\left[Y_{q,\beta}(R)+Z_{r,\gamma}(R)\right]<+\infty;\\
&\mathrm{(A2)}\;p\in\left(\frac{3}{2},3\right),\;\liminf\limits_{R\rightarrow+\infty}\left[X_{p,\alpha}(R)+Y_{q,\beta}(R)+Z_{r,\gamma}(R)\right]<+\infty.
\end{align*}
Then $u=v=0$ and $\theta=0$.
\end{Thm}

Roughly speaking, we can improve Theorem \ref{main1} by relaxing the growth conditions to incorporate logarithmic factors.
\begin{Thm}\label{main2}
Let $(u,\pi,v,\theta)$ be a smooth solution of \eqref{equ1.1}. Suppose the parameters  satisfy {\rm Assumptions \ref{a1.1}, \ref{a1.2}, \ref{a1.3}, \ref{a1.4}} and {\rm\ref{a1.5}}.
Furthermore, assume that one of the following assumptions holds
\begin{align*}
&\mathrm{(B1)}\;p\in\left[3,\frac{9}{2}\right],\;\lim\limits_{R\rightarrow+\infty}X_{p,\alpha}(R)=0,\;\limsup\limits_{R\rightarrow+\infty}\left[Y_{q,\beta,\mu}(R)+Z_{r,\gamma,\nu}(R)\right]<+\infty;\\
&\mathrm{(B2)}\;p\in\left(\frac{3}{2},3\right),\;\limsup\limits_{R\rightarrow+\infty}\left[X_{p,\alpha,\lambda}(R)+Y_{q,\beta,\mu}(R)+Z_{r,\gamma,\nu}(R)\right]<+\infty.
\end{align*}
Then $u=v=0$ and $\theta=0$.
\end{Thm}

\begin{Rem}
The first baroclinic mode of vector velocity $v$ and the temperature $\theta$ possess symmetric positions in {\rm Theorem \ref{main1}} while this symmetry is broken in {\rm Theorem \ref{main2}}. For the reason, see {\rm Remark \ref{Rem4.1}}.
\end{Rem}

As a consequence of Theorem \ref{main1} and Remark \ref{Rem3.1} in Section \ref{sec3}, we obtain the following Liouville type theorems in Lebesgue spaces.
\begin{Cor}\label{Cor1.1}
Let $(u,\pi,v,\theta)$ be a smooth solution of \eqref{equ1.1} and $(p,q,r)\in\left(\frac{3}{2},\frac{9}{2}\right]\times[1,6]\times[1,6]$. Suppose $u\in L^p(\mathbb{R}^3)$, $v\in L^q(\mathbb{R}^3)$, $\theta\in L^r(\mathbb{R}^3)$, where $p,q,r$ satisfy
\begin{equation}\label{ine1.7}
\frac{1}{p}+\frac{2}{q}\geq\frac{2}{3},\;\frac{1}{p}+\frac{2}{r}\geq\frac{2}{3},\;\frac{1}{q}+\frac{1}{r}\geq\frac{2}{3}.
\end{equation}
Then $u=v=0$ and $\theta=0$.
\end{Cor}

\begin{Rem}
It is not difficult to verify that if $(p,q,r)$ satisfies one of the following cases
\begin{align*}
&\mathrm{(i)}\;(p,q,r)\in \left(\frac{3}{2},3\right]\times[1,q_0]\times[1,r_0], \text{ where }q_0,r_0\in[2,6],\;\frac{1}{q_0}+\frac{1}{r_0}=\frac{2}{3};\\
&\mathrm{(ii)}\;(p,q,r)\in\left[3,\frac{9}{2}\right]\times[1,2p']\times[1,2p'],
\end{align*}
\eqref{ine1.7} fulfils automatically.
In {\rm\cite[Theorem 1.1]{HD24}}, Ding and Wu showed the Liouville type theorem for \eqref{equ1.1} under the condition $u,v,\theta\in L^p(\mathbb{R}^3)$ with $p\in\left(\frac{3}{2},3\right)$, which is a special example of case $\mathrm{(i)}$.
\end{Rem}

As for Liouville type theorems under the decay conditions, we only take some endpoint cases for instance.
\begin{Cor}\label{Cor1.2}
Let $(u,\pi,v,\theta)$ be a smooth solution of \eqref{equ1.1}. Then Liouville type theorem for \eqref{equ1.1} holds under one of the following decay conditions
\begin{align*}
\mathrm{(C1)}&\;u(x)=o(|x|^{-\frac{2}{3}}),\;v(x)=O(|x|^{-\frac{2}{3}}),\; \theta(x)=o(|x|^{-\frac{4}{3}})\text{ as }|x|\rightarrow+\infty;\\
\mathrm{(C2)}&\;u(x)=o(|x|^{-\frac{2}{3}}),\;v(x)=o(|x|^{-\frac{2}{3}}),\; \theta(x)=O(|x|^{-\frac{4}{3}})\text{ as }|x|\rightarrow+\infty;\\
\mathrm{(C3)}&\;u(x)=o(|x|^{-\frac{2}{3}}),\;v(x)=O(|x|^{-\frac{4}{3}}),\; \theta(x)=o(|x|^{-\frac{2}{3}})\text{ as }|x|\rightarrow+\infty;\\
\mathrm{(C4)}&\;u(x)=o(|x|^{-\frac{2}{3}}),\;v(x)=o(|x|^{-\frac{4}{3}}),\; \theta(x)=O(|x|^{-\frac{2}{3}})\text{ as }|x|\rightarrow+\infty.
\end{align*}
\end{Cor}

\subsection{Liouville type theorems under some assumptions imposed on $u,v,\nabla\theta$}
Before showing our second class of Liouville type results, we introduce some basic assumptions on the parameters $p,q,r$ and $\alpha$, $\beta$, $\gamma$.
\begin{Assum}\label{a1.6}
Let $(p,q,r)\in\left(\frac{3}{2},\frac{9}{2}\right]\times[1,6]\times[1,2]$ and $\alpha\in\left[0,\alpha_p\right]$, $\beta\in\left[0,\frac{3}{q}-\frac{1}{2}\right]$, $\gamma\in\left[0,\frac{3}{r}-\frac{3}{2}\right]$, where $\alpha_p$ is defined in {\rm Assumption \ref{a1.1}}.
\end{Assum}

\begin{Assum}\label{a1.7}
\begin{itemize}
\item[(i)] This item is the same as {\rm Assumption \ref{a1.2}(i)}.
\item[(ii)] When $(p,r)\in\left(\frac{3}{2},\frac{9}{2}\right]\times[1,2]$, it holds that
$$
\alpha+\frac{(4p-6)\min\{r,\frac{6p}{5p-3}\}}{\left(6-3\min\{r,\frac{6p}{5p-3}\}\right)p}\gamma\leq\frac{3}{p}+\frac{6}{\max\{r,\frac{6p}{5p-3}\}}-4.
$$
In addition, when $(p,r)\in\left(\frac{3}{2},3\right)\times\left[\frac{6p}{5p-3},2\right]$, equality can not be achieved in the above inequality.
\item[(iii)] When $(q,r)\in [1,6]\times[1,2]$, it holds that
$$
\frac{2\min\{q,2\}}{6-\min\{q,2\}}\beta+\frac{2\min\{r,\frac{6}{5}\}}{6-3\min\{r,\frac{6}{5}\}}\gamma\leq\frac{3}{\max\{q,2\}}+\frac{3}{\max\{r,\frac{6}{5}\}}-3.
$$
In addition, when $(q,r)\in[2,6]\times\left[\frac{6}{5},2\right]$, equality can not be achieved in the above inequality.
\end{itemize}		
\end{Assum}

\begin{Assum}\label{a1.8}
\begin{itemize}
		\item[(i)] This item is the same as {\rm Assumption \ref{a1.4}(i)}.
        \item[(ii)] When $r\in[1,\frac{6p}{5p-3})$, it holds tha
$$\lambda+\frac{(4p-6)r}{(6-3r)p}\nu\leq\frac{6p-(5p-3)r}{(6-3r)p}.$$
        \item[(iii)] When $(q,r)\in [1,6]\times[1,\frac{6}{5})$, or $[1,2)\times[\frac{6}{5},2]$,  it holds that
$$
\frac{2\min\{q,2\}}{6-\min\{q,2\}}\mu+\frac{2\min\{r,\frac{6}{5}\}}{6-3\min\{r,\frac{6}{5}\}}\nu\leq\frac{6-3\min\{q,2\}}{12-2\min\{q,2\}}+\frac{6-5\min\{r,\frac{6}{5}\}}{12-6\min\{r,\frac{6}{5}\}}.
$$
\end{itemize}		
\end{Assum}

\begin{Assum}\label{a1.9}
\begin{itemize}
	    \item[(i)]{\rm(ii)\;(iii)} These three items are the same as {\rm (i)\;(ii)\;(iii)} in {\rm  Assumption \ref{a1.5}}, respectively.
        \item[(iv)] When $1\leq q<2$, we require $(r,\beta,\gamma)\neq(\frac{6}{5},\frac{3}{q}-\frac{1}{2},0)$.
        \item[(v)] When $1\leq r<\frac{6}{5}$, we require $(q,\beta,\gamma)\neq(2,0,\frac{3}{r}-\frac{3}{2})$.
\end{itemize}
\end{Assum}

In contrast to Theorem \ref{main1} and Theorem \ref{main2}, we impose some growth conditions on Lebesgue norms of $\nabla\theta$ over the annulus.
\begin{Thm}\label{main3}
Let $(u,\pi,v,\theta)$ be a smooth solution of \eqref{equ1.1}.
Suppose the parameters satisfy {\rm Assumptions \ref{a1.6}} and {\rm\ref{a1.7}}. Furthermore, assume that one of the following assumptions holds
\begin{align*}
&\mathrm{(A1)'}\;p\in\left[3,\frac{9}{2}\right],\;\liminf\limits_{R\rightarrow+\infty}X_{p,\alpha}(R)=0,\;\limsup\limits_{R\rightarrow+\infty}\left[Y_{q,\beta}(R)+Z'_{r,\gamma}(R)\right]<+\infty;\\
&\mathrm{(A2)'}\;p\in\left(\frac{3}{2},3\right),\;\liminf\limits_{R\rightarrow+\infty}\left[X_{p,\alpha}(R)+Y_{q,\beta}(R)+Z'_{r,\gamma}(R)\right]<+\infty.
\end{align*}
Then $u=v=0$ and $\theta$ is a constant.
\end{Thm}

\begin{Thm}\label{main4}
Let $(u,\pi,v,\theta)$ be a smooth solution of \eqref{equ1.1}. Suppose the parameters satisfy {\rm Assumptions \ref{a1.6}, \ref{a1.7}, \ref{a1.3}, \ref{a1.8}} and {\rm\ref{a1.9}}.
Furthermore, assume that one of the following assumptions holds
\begin{align*}
&\mathrm{(B1)'}\;p\in\left[3,\frac{9}{2}\right],\;\lim\limits_{R\rightarrow+\infty}X_{p,\alpha}(R)=0,\;\limsup\limits_{R\rightarrow+\infty}\left[Y_{q,\beta,\mu}(R)+Z'_{r,\gamma,\nu}(R)\right]<+\infty;\\
&\mathrm{(B2)'}\;p\in\left(\frac{3}{2},3\right),\;\limsup\limits_{R\rightarrow+\infty}\left[X_{p,\alpha,\lambda}(R)+Y_{q,\beta,\mu}(R)+Z'_{r,\gamma,\nu}(R)\right]<+\infty.
\end{align*}
Then $u=v=0$ and $\theta$ is a constant.
\end{Thm}

Similarly to Corollary \ref{Cor1.1}, we have

\begin{Cor}
Let $(u,\pi,v,\theta)$ be a smooth solution of \eqref{equ1.1} and $(p,q,r)\in\left(\frac{3}{2},\frac{9}{2}\right]\times[1,6]\times[1,2]$. Suppose $u\in L^p(\mathbb{R}^3)$, $v\in L^q(\mathbb{R}^3)$, $\nabla\theta\in L^r(\mathbb{R}^3)$, where $p,q,r$ satisfy
\begin{equation}\label{ine1.8}
\frac{1}{p}+\frac{2}{q}\geq\frac{2}{3},\;\frac{1}{p}+\frac{2}{r}\geq\frac{4}{3},\;\frac{1}{q}+\frac{1}{r}\geq1.
\end{equation}
Then $u=v=0$ and $\theta$ is a constant.
\end{Cor}

\begin{Rem}
It is not difficult to verify that if $(p,q,r)$ satisfies one of the following cases
\begin{align*}
&\mathrm{(i)}\;(p,q,r)\in \left(\frac{3}{2},3\right]\times[1,q_0]\times[1,r_0], \text{ where }q_0\in[2,6],\;r_0\in\left[\frac{6}{5},2\right],\;\frac{1}{q_0}+\frac{1}{r_0}=1;\\
&\mathrm{(ii)}\;(p,q,r)\in\left[3,\frac{9}{2}\right]\times[1,2p']\times\left[1,\frac{6p}{5p-3}\right],
\end{align*}
\eqref{ine1.8} fulfils automatically.
Cho et al. {\rm\cite{CIY23}} established Liouville type result for \eqref{equ1.1} under the assumption $u\in L^3(\mathbb{R}^3)$ and $v,\nabla\theta\in L^2(\mathbb{R}^3)$, which is a special example of case $\mathrm{(i)}$.
\end{Rem}

Similarly to Corollary \ref{Cor1.2}, we have
\begin{Cor}
Let $(u,\pi,v,\theta)$ be a smooth solution of \eqref{equ1.1}. Then Liouville type theorem for \eqref{equ1.1} holds under one of the following decay conditions
\begin{align*}
\mathrm{(C1)'}&\;u(x)=o(|x|^{-\frac{2}{3}}),\;v(x)=O(|x|^{-\frac{2}{3}}),\; \nabla\theta(x)=o(|x|^{-\frac{7}{3}})\text{ as }|x|\rightarrow+\infty;\\
\mathrm{(C2)'}&\;u(x)=o(|x|^{-\frac{2}{3}}),\;v(x)=o(|x|^{-\frac{2}{3}}),\; \nabla\theta(x)=O(|x|^{-\frac{7}{3}})\text{ as }|x|\rightarrow+\infty;\\
\mathrm{(C3)'}&\;u(x)=o(|x|^{-\frac{2}{3}}),\;v(x)=O(|x|^{-\frac{4}{3}}),\; \nabla\theta(x)=o(|x|^{-\frac{5}{3}})\text{ as }|x|\rightarrow+\infty;\\
\mathrm{(C4)'}&\;u(x)=o(|x|^{-\frac{2}{3}}),\;v(x)=o(|x|^{-\frac{4}{3}}),\; \nabla\theta(x)=O(|x|^{-\frac{5}{3}})\text{ as }|x|\rightarrow+\infty.
\end{align*}
\end{Cor}

Finally, we give an arrangement of the rest of this paper. In Section \ref{sec2}, we recall the property of the Bogovskii operator, the Poincar\'{e}-Sobolev inequality, the Poincar\'{e} inequality and a useful iteration lemma. Section \ref{sec3} is devoted to proving Theorem \ref{main1}, Corollary \ref{Cor1.1} and Corollary \ref{Cor1.2} while Section \ref{sec4} is devoted to proving Theorem \ref{main2}. The proofs of Theorem \ref{main3} and Theorem \ref{main4} are presented in Section \ref{sec5}.
Throughout this article, we use $C$ to denote a finite inessential constant which may be different from line to line.
\v0.1in
\section{Preliminaries}\label{sec2}

The first lemma is the existence and boundedness of the Bogovskii map, which is used to deal with the pressure term $\nabla \pi$.
\begin{Lem}\label{Lem2.1}$($See {\rm \cite[Lemma 1]{Tsai21}}$)$
Let $E$ be a bounded Lipschitz domain in $\mathbb{R}^{3}$. Denote $L_{0}^{\sigma}(E):=\{f\in L^{\sigma}(E):\int_E fdx=0\}$ with $1<\sigma<\infty$. There exists a linear operator
\begin{equation*}
\mathrm{Bog}:L_{0}^{\sigma}(E)\rightarrow W_{0}^{1,\sigma}(E),
\end{equation*}
such that for any  $f\in L_{0}^{\sigma}(E),w=\mathrm{Bog}f$ is a vector field satisfying
\begin{equation*}
w\in W_{0}^{1,\sigma}(E), \hspace{0.3cm}\mathrm{div}w=f,\hspace{0.3cm}\|\nabla w\|_{L^{\sigma}(E)}\leq C_{\mathrm{Bog}}(E,\sigma)\|f\|_{L^{\sigma}(E)},
\end{equation*}
where the constant $C_{\mathrm{Bog}}(E,\sigma)$  is independent of $f$. By using a rescaling argument, we see
$$C_{\mathrm{Bog}}(RE,\sigma)=C_{\mathrm{Bog}}(E,\sigma),\text{ where $RE=\{Rx:x\in E\}$ and $R>0$.}$$
\end{Lem}

The second lemma is the so-called Poincar\'{e}-Sobolev inequality. Let $(f)_{\Omega}$ represent the mean value of the function $f$ over $\Omega$.
\begin{Lem}\label{Lem2.2}$($See {\rm\cite[Theorem 3.15]{Giusti}} and {\rm \cite[Remark 2]{CIY23}}$)$
Let $\Omega\subset\mathbb{R}^n$ be a bounded Lipschitz domain. Assume $1\leq p<n$. Then there exists a positive constant $C(n,p,\Omega)$ such that
$$\|f-(f)_{\Omega}\|_{L^{\frac{np}{n-p}}(\Omega)}\leq C(n,p,\Omega)\|\nabla f\|_{L^{p}(\Omega)}\text{ for every $f\in W^{1,p}(\Omega)$.}$$
By using a rescaling argument, we see that the constant $C(n,p,R\Omega)$ does not depend on $R$, i.e.,  $C(n,p,R\Omega)=C(n,p,\Omega)$.
\end{Lem}

The third lemma is the so-called Poincar\'{e} inequality.
\begin{Lem}\label{Lem2.2a}$($See {\rm\cite[p.292, Theorem 1]{Evans}}$)$
Let $\Omega\subset\mathbb{R}^n$ be a bounded $C^1$ domain. Assume $1\leq p\leq +\infty$. Then there exists a positive constant $C(n,p,\Omega)$ such that
$$\|f-(f)_{\Omega}\|_{L^p(\Omega)}\leq C(n,p,\Omega)\|\nabla f\|_{L^{p}(\Omega)}\text{ for every $f\in W^{1,p}(\Omega)$.}$$
By using a rescaling argument, we see that $C(n,p,R\Omega)=RC(n,p,\Omega)$.
\end{Lem}

Next, we show the following standard iteration lemma, which is a generalization of \cite[Lemma 3.1]{Giaquinta}, and can be found in \cite[Lemma 2.1]{CL24}.
\begin{Lem}\label{Lem2.3}
Let $f(t)$ be a non-negative bounded function on $\left[r_0, r_1\right] \subset \mathbb{R}^{+}$. If there are non-negative constants $a_i, b_i,\alpha_i$, $i=1,2,\cdots,m$, and a parameter $\theta_0 \in[0,1)$ such that for any $r_0 \leq s<t \leq r_1$, it holds that
$$
f(s) \leq \theta_0 f(t)+\sum_{i=1}^m\left(\frac{a_i}{(t-s)^{\alpha_i}}+b_i\right),
$$
then
$$
f(s) \leq C\sum_{i=1}^m\left(\frac{a_i}{(t-s)^{\alpha_i}}+b_i\right),
$$
where $C$ is a constant depending on $\alpha_1,\alpha_2,\cdots,\alpha_m$ and $\theta_0$.
\end{Lem}

\section{Proofs of Theorem \ref{main1}, Corollary \ref{Cor1.1} and Corollary \ref{Cor1.2}}\label{sec3}
For any $\rho>0$, we denote
\begin{equation}\label{ine3.1}
f(\rho)=\|u\|_{L^6 (B_\rho)}^2+\|v\|_{L^6 (B_\rho)}^2+\|\theta\|_{L^6 (B_\rho)}^2+\|\nabla u\|_{L^2 (B_\rho)}^2+\|\nabla v\|_{L^2 (B_\rho)}^2+\|\nabla \theta\|_{L^2 (B_\rho)}^2.
\end{equation}
For any $\sqrt{3}R\leq s<t\leq 2R$, we introduce several integrals involving $u$, $v$ and $\theta$. Denote
\begin{align}\label{ine3.2}
&J_1=\frac{1}{(t-s)^2}\int_{B_t \backslash B_s}(|u|^2+|v|^2+|\theta|^2) d x,\;J_2=\frac{1}{t-s} \int_{B_t \backslash B_\frac{3R}{2}}|u|^3  d x,\notag\\
&J_3=\frac{1}{t-s}\|v\|_{L^{2p'}(B_t\backslash B_\frac{3R}{2})}^{2}\|u\|_{L^p(B_t\backslash B_s)},\;J_4=\frac{1}{t-s}\|\theta\|_{L^{2p'}(B_t\backslash B_\frac{3R}{2})}^{2}\|u\|_{L^p(B_t\backslash B_s)},\\
&J_5=\frac{1}{t-s}\int_{B_t \backslash B_s}|v||\theta|dx.\notag
\end{align}

In order to prove Theorem \ref{main1}, we first establish an important energy estimate.
\begin{Lem}\label{Lem3.1}
Let $(u,\pi,v,\theta)$ be a smooth solution of \eqref{equ1.1} and $\sqrt{3}R\leq s<t\leq 2R$. Then we have
\begin{align}\label{ine3.3}
f(s)\leq\frac{1}{2}\int_{B_t\backslash B_\frac{3R}{2}}|\nabla u|^2dx+C(J_1+J_2+J_3+J_4+J_5),
\end{align}
where $f$ is defined by \eqref{ine3.1} and $J_1,J_2,J_3,J_4,J_5$ are defined by \eqref{ine3.2}.
\end{Lem}
\begin{proof}
Since $\sqrt{3}R\leq s<t\leq 2R$, we infer $s\geq\frac{\sqrt{3}t}{2}>\frac{3R}{2}$. We introduce a cut-off function $\eta \in C_0^{\infty}\left(\mathbb{R}^{3}\right)$ satisfying
\begin{equation*}
\eta(x)= \begin{cases}1, & |x| <s, \\ 0, & |x| >\frac{s+t}{2},\end{cases}
\end{equation*}
with
$$\text{$0\leq\eta (x)\leq 1$, and $\|\nabla \eta\|_{L^{\infty}} \leq \frac{C}{t-s}$, $\|\nabla^2\eta\|_{L^{\infty}} \leq \frac{C}{(t-s)^2}$.}$$
By Lemma \ref{Lem2.1}, there exists $w\in W_{0}^{1,\sigma}(B_t\backslash \overline{B_\frac{\sqrt{3}t}{2}})$ such that $w$ satisfies
$$
\mathrm{div }w=u\cdot\nabla\eta^{2} \text{ in }B_t\backslash \overline{B_\frac{\sqrt{3}t}{2}},
$$
with the estimate
\begin{equation}\label{ine3.4}
\aligned
\|\nabla w\|_{L^\sigma\left(B_t\backslash \overline{B_\frac{\sqrt{3}t}{2}}\right)}\leq C\|u\cdot\nabla\eta^2\|_{L^\sigma\left(B_t\backslash \overline{B_\frac{\sqrt{3}t}{2}}\right)}\leq\frac{C}{t-s}\|u\|_{L^\sigma(B_t\backslash B_s)},
\endaligned
\end{equation}
for any $1<\sigma<+\infty$. We extend $w$ by zero to $B_\frac{\sqrt{3}t}{2}$, then $w\in W_{0}^{1,\sigma}(B_t).$

Multiply both sides of $\eqref{equ1.1}_{1}$, $\eqref{equ1.1}_{2}$ and $\eqref{equ1.1}_{3}$ by $u \eta^2-w$, $v \eta^2$ and $\theta \eta^2$ respectively, integrate over $B_t$ and apply integration by parts. This procedure yields
\begin{align}
&\int_{B_t}\left(|\nabla u|^{2}+|\nabla v|^{2}+|\nabla \theta|^{2}\right)\eta^2 d x\notag\\
=&\frac{1}{2}\int_{B_t}(|u|^2+|v|^2+|\theta|^2) \Delta \eta^2 d x+\frac{1}{2} \int_{B_t}(|u|^2+|v|^2+|\theta|^2) u \cdot \nabla \eta^2 d x\notag\\
&+\int_{B_t}\nabla u:\nabla w dx+ \int_{B_t}(v\otimes v):(u\otimes\nabla \eta^2) dx+ \int_{B_t}\theta v\cdot\nabla \eta^2 dx\label{ine3.5}\\
& - \int_{B_t}(u\otimes u+v\otimes v): \nabla w  dx.\notag
\end{align}
Applying the Gagliardo-Nirenberg inequality, we have
\begin{align*}
&\|u \eta\|_{L^6 (B_t)}^2+\|v \eta\|_{L^6 (B_t)}^2+\|\theta \eta\|_{L^6 (B_t)}^2\\
\leq&C\left(\|\nabla(u \eta)\|_{L^2 (B_t)}^2+\|\nabla(v \eta)\|_{L^2 (B_t)}^2+\|\nabla(\theta \eta)\|_{L^2 (B_t)}^2 \right)\\
\leq&C \left(\|\eta\nabla u\|_{L^2 (B_t)}^2+\|\eta\nabla v\|_{L^2 (B_t)}^2+\|\eta\nabla \theta\|_{L^2 (B_t)}^2\right)\\
&+C\left(\|u\otimes\nabla \eta\|_{L^2 (B_t)}^2+\|v\otimes\nabla \eta\|_{L^2 (B_t)}^2+\|\theta\otimes\nabla \eta\|_{L^2 (B_t)}^2 \right).
\end{align*}
Combining the above inequality and \eqref{ine3.5}, we have
\begin{align}
&\|u\eta\|_{L^6 (B_t)}^2+\|v \eta\|_{L^6 (B_t)}^2+\|\theta \eta\|_{L^6 (B_t)}^2+\int_{B_t}\left(|\nabla u|^{2}+|\nabla v|^{2}+|\nabla \theta|^{2}\right)\eta^2 d x\notag\\
\leq&\frac{C}{(t-s)^2}\int_{B_t \backslash B_s}(|u|^2+|v|^2+|\theta|^2) d x+ \frac{C}{t-s} \int_{B_t \backslash B_s}|u|^3  d x\notag\\
&+C\int_{B_t\backslash B_\frac{\sqrt{3}t}{2}}|\nabla u||\nabla w| dx+C\int_{B_t\backslash B_\frac{\sqrt{3}t}{2}}|u|^{2} |\nabla w|dx+C\int_{B_t\backslash B_\frac{\sqrt{3}t}{2}}|v|^{2}|\nabla w|dx \label{ine3.6}\\
&+\frac{C}{t-s}\int_{B_t \backslash B_s}|v|^2|u|dx+\frac{C}{t-s}\int_{B_t \backslash B_s}|\theta|^2|u|dx+\frac{C}{t-s}\int_{B_t \backslash B_s}|v||\theta|dx\notag\\
=&:I_1+I_2+I_3+I_4+I_5+I_6+I_7+I_8.\notag
\end{align}

Obviously, we have
\begin{align}\label{ine3.7}
I_1=CJ_1,\;I_8=CJ_5.
\end{align}
By the H\"{o}lder inequality and \eqref{ine3.4}, we obtain
\begin{align}\label{ine3.8}
I_2+I_4\leq&\frac{C}{t-s}\|u\|_{L^{3}(B_t\backslash B_s)}^3+C\|u\|_{L^{3}(B_t\backslash B_\frac{\sqrt{3}t}{2})}^{2}\|\nabla w\|_{L^{3}(B_t\backslash B_\frac{\sqrt{3}t}{2})}\notag\\
\leq&\frac{C}{t-s}\|u\|_{L^{3}(B_t\backslash B_s)}^3+\frac{C}{t-s}\|u\|_{L^{3}(B_t\backslash B_\frac{\sqrt{3}t}{2})}^{2}\|u\|_{L^{3}(B_t\backslash B_s)}\\
\leq&\frac{C}{t-s}\|u\|_{L^{3}(B_t\backslash B_\frac{3R}{2})}^3=CJ_2.\notag
\end{align}
By the Young inequality and \eqref{ine3.4}, we have
\begin{align}\label{ine3.9}
I_3&\leq\frac{1}{2}\int_{B_t\backslash B_\frac{\sqrt{3}t}{2}}|\nabla u|^2dx+C\int_{B_t\backslash B_\frac{\sqrt{3}t}{2}}|\nabla w|^2dx\notag\\
&\leq\frac{1}{2}\int_{B_t\backslash B_\frac{\sqrt{3}t}{2}}|\nabla u|^2dx+\frac{C}{(t-s)^2}\int_{B_t \backslash B_s}|u|^2dx\\
&\leq\frac{1}{2}\int_{B_t\backslash B_\frac{3R}{2}}|\nabla u|^2dx+CJ_1.\notag
\end{align}
Using the H\"{o}lder inequality and \eqref{ine3.4} again, we obtain
\begin{align}\label{ine3.10}
I_5+I_6\leq&C\|v\|_{L^{2p'}(B_t\backslash B_\frac{\sqrt{3}t}{2})}^{2}\|\nabla w\|_{L^p(B_t\backslash B_\frac{\sqrt{3}t}{2})}+\frac{C}{t-s}\|v\|_{L^{2p'}(B_t\backslash B_s)}^{2}\|u\|_{L^p(B_t\backslash B_s)}\notag\\
\leq&\frac{C}{t-s}\|v\|_{L^{2p'}(B_t\backslash B_\frac{\sqrt{3}t}{2})}^{2}\|u\|_{L^p(B_t\backslash B_s)}+\frac{C}{t-s}\|v\|_{L^{2p'}(B_t\backslash B_s)}^{2}\|u\|_{L^p(B_t\backslash B_s)}\\
\leq&\frac{C}{t-s}\|v\|_{L^{2p'}(B_t\backslash B_\frac{3R}{2})}^{2}\|u\|_{L^p(B_t\backslash B_s)}=CJ_3.\notag
\end{align}
Using the H\"{o}lder inequality, we obtain
\begin{align}\label{ine3.11}
I_7\leq&\frac{C}{t-s}\|\theta\|_{L^{2p'}(B_t\backslash B_s)}^{2}\|u\|_{L^p(B_t\backslash B_s)}\leq\frac{C}{t-s}\|\theta\|_{L^{2p'}(B_t\backslash B_\frac{3R}{2})}^{2}\|u\|_{L^p(B_t\backslash B_s)}=CJ_4.
\end{align}
Collecting \eqref{ine3.6}, \eqref{ine3.7}, \eqref{ine3.8}, \eqref{ine3.9}, \eqref{ine3.10} and \eqref{ine3.11}, we conclude that \eqref{ine3.3} holds.
\end{proof}

The estimates on $J_1$, $J_2$, $J_3$, $J_4$, $J_5$ are given in the next five lemmas. From the proof of Theorem \ref{main1}, we can conclude that the restrictions on the parameters $\alpha,\beta,\gamma$ are due to these estimates. The relation between $\alpha$ and $\beta$ is determined by the estimate of $J_3$. The relation between $\alpha$ and $\gamma$ is determined by the estimate of $J_4$. The relation between $\beta$ and $\gamma$ is determined by the estimate of $J_5$.
\begin{Lem}\label{Lem2.4}$($See {\rm\cite[Lemma 2.4]{ZZB25}}$)$
Let $\sqrt{3}R\leq s<t\leq 2R$ and $p,q,r\geq1$.
Suppose that $u,v,\theta$ are smooth vector-valued functions. Then
\begin{itemize}
\item[(i)] For any $\varepsilon>0$, there exist positive constants $C_\varepsilon$ and $C$ such that
\begin{align}
J_1\leq&\varepsilon\|u\|_{L^{6}(B_t\backslash B_s)}^{2}+\frac{C_\varepsilon}{(t-s)^{\frac{6}{p}-1}}\|u\|_{L^{p}(A_{R})}^{2}+\frac{C}{(t-s)^2}R^{3-\frac{6}{p}}\|u\|_{L^p (A_{R})}^2\notag\\
&+\varepsilon\|v\|_{L^{6}(B_t\backslash B_s)}^{2}+\frac{C_\varepsilon}{(t-s)^{\frac{6}{q}-1}}\|v\|_{L^{q}(A_{R})}^{2}+\frac{C}{(t-s)^2}R^{3-\frac{6}{q}}\|v\|_{L^{q}(A_{R})}^{2}\label{ine3.12}\\
&+\varepsilon\|\theta\|_{L^{6}(B_t\backslash B_s)}^{2}+\frac{C_\varepsilon}{(t-s)^{\frac{6}{r}-1}}\|\theta\|_{L^{r}(A_{R})}^{2}+\frac{C}{(t-s)^2}R^{3-\frac{6}{r}}\|\theta\|_{L^{r}(A_{R})}^{2}.\notag
\end{align}
\item[(ii)] It holds that
\begin{equation}\label{ine3.13}
\aligned
J_1\leq\frac{CR^2}{(t-s)^2}\left(\|u\|_{L^6(A_{R})}^2 +\|v\|_{L^6(A_{R})}^2+\|\theta\|_{L^6(A_{R})}^2\right).
\endaligned
\end{equation}
\end{itemize}
\end{Lem}

\begin{Lem}$\label{Lem2.5}($See {\rm\cite[Lemma 2.5]{ZZB25}}$)$
Let $\sqrt{3}R\leq s<t\leq 2R$. Suppose that $u$ is a smooth vector-valued function. Then we have the following conclusions:
\begin{itemize}
\item[(i)]  Assume $p\geq3$. It holds that
\begin{equation}\label{ine3.14}
J_2\leq\frac{C}{t-s}R^{3-\frac{9}{p}}\|u\|_{L^{p}(A_{R})}^{3}.
\end{equation}
\item[(ii)]  Assume $p\in\left[1,3\right)$. It holds that
\begin{equation}\label{ine3.15}
J_2\leq \frac{1}{t-s}\|u\|_{L^p(B_t\backslash B_\frac{3R}{2})}^{\frac{3p}{6-p}}\|u\|_{L^6(B_t\backslash B_\frac{3R}{2})}^{\frac{18-6p}{6-p}}.
\end{equation}
\item[(iii)] Assume $p\in\left(\frac{3}{2},3\right)$. For any $\varepsilon>0$, it holds that
\begin{equation}\label{ine3.16}
\aligned
J_2&\leq\varepsilon\|u\|_{L^6(B_t\backslash B_\frac{3R}{2})}^2+\frac{C_\varepsilon}{(t-s)^\frac{6-p}{2p-3}}\|u\|_{L^p(A_{R})}^\frac{3p}{2p-3}.
\endaligned
\end{equation}
\end{itemize}
\end{Lem}

\begin{Lem}\label{Lem2.7}$($See {\rm\cite[Lemma 2.6]{ZZB25}}$)$
Let $\sqrt{3}R\leq s<t\leq 2R$. Suppose that $u,v$ are smooth vector-valued functions. Then
\begin{itemize}
\item[(i)] Let $\frac{3}{2}<p\leq\frac{9}{2}$, $1\leq q<2p'$. For any $\varepsilon>0$, there exists a positive constant $C_\varepsilon$ such that
\begin{equation}\label{ine3.17}
J_3\leq\varepsilon\|v\|_{L^6(B_t\backslash B_\frac{3R}{2})}^2+\frac{C_\varepsilon}{(t-s)^\frac{(6-q)p'}{(3-p')q}}\|u\|_{L^p(A_{R})}^\frac{(6-q)p'}{(3-p')q}\|v\|_{L^q(A_{R})}^2.
\end{equation}
\item[(ii)] Let $\frac{3}{2}<p\leq\frac{9}{2}$, $1\leq q<2p'$. It holds that
\begin{equation}\label{ine3.18}
\aligned
J_3&\leq\frac{C}{t-s}\|u\|_{L^p(B_t\backslash B_s)}\|v\|_{L^q(B_t\backslash B_\frac{3R}{2})}^\frac{2(3-p')q}{(6-q)p'}\|v\|_{L^6(B_t\backslash B_\frac{3R}{2})}^\frac{12p'-6q}{(6-q)p'}.
\endaligned
\end{equation}
\item[(iii)] Let $\frac{3}{2}<p\leq\frac{9}{2}$, $q\geq 2p'$. It holds that
\begin{equation}\label{ine3.19}
\aligned
J_3&\leq\frac{C}{t-s}R^{3-\frac{3}{p}-\frac{6}{q}}\|v\|_{L^q(A_{R})}^{2}\|u\|_{L^p(A_{R})}.
\endaligned
\end{equation}
\end{itemize}

\end{Lem}

Similarly to Lemma \ref{Lem2.7}, we have the following lemma.
\begin{Lem}
Let $\sqrt{3}R\leq s<t\leq 2R$. Suppose that $u$ is a smooth vector-valued function and $\theta$ is a smooth function. Then
\begin{itemize}
\item[(i)] Let $\frac{3}{2}<p\leq\frac{9}{2}$, $1\leq r<2p'$. For any $\varepsilon>0$, there exists a positive constant $C_\varepsilon$ such that
\begin{equation}\label{ine3.20}
J_4\leq\varepsilon\|\theta\|_{L^6(B_t\backslash B_\frac{3R}{2})}^2+\frac{C_\varepsilon}{(t-s)^\frac{(6-r)p'}{(3-p')r}}\|u\|_{L^p(A_{R})}^\frac{(6-r)p'}{(3-p')r}\|\theta\|_{L^r(A_{R})}^2.
\end{equation}
\item[(ii)] Let $\frac{3}{2}<p\leq\frac{9}{2}$, $1\leq r<2p'$. It holds that
\begin{equation}\label{ine3.21}
\aligned
J_4&\leq\frac{1}{t-s}\|u\|_{L^p(B_t\backslash B_s)}\|\theta\|_{L^r(B_t\backslash B_\frac{3R}{2})}^\frac{2(3-p')r}{(6-r)p'}\|\theta\|_{L^6(B_t\backslash B_\frac{3R}{2})}^\frac{12p'-6r}{(6-r)p'}.
\endaligned
\end{equation}
\item[(iii)] Let $\frac{3}{2}<p\leq\frac{9}{2}$, $r\geq 2p'$. It holds that
\begin{equation}\label{ine3.22}
\aligned
J_4&\leq\frac{C}{t-s}R^{3-\frac{3}{p}-\frac{6}{r}}\|\theta\|_{L^r(A_{R})}^{2}\|u\|_{L^p(A_{R})}.
\endaligned
\end{equation}
\end{itemize}

\end{Lem}

\begin{Lem}\label{Lem2.9}
Let $\sqrt{3}R\leq s<t\leq 2R$. Suppose that $v$ is a smooth vector-valued function and $\theta$ is a smooth function. Then
\begin{itemize}
\item[(i)] Let $1\leq q<2$, $1\leq r<2$. For any $\varepsilon>0$, there exists a positive constant $C_\varepsilon$ such that
\begin{equation}\label{ine3.23}
\aligned
J_5\leq&\varepsilon\|v\|_{L^6(B_t\backslash B_\frac{3R}{2})}^2+\varepsilon\|\theta\|_{L^6(B_t\backslash B_\frac{3R}{2})}^2+\\
&\frac{C_{\varepsilon}}{(t-s)^{\frac{(6-q)(6-r)}{q(6-r)+r(6-q)}}}\|v\|_{L^q(A_{R})}^{\frac{2q(6-r)}{q(6-r)+r(6-q)}}\|\theta\|_{L^r(A_{R})}^{\frac{2r(6-q)}{q(6-r)+r(6-q)}}.
\endaligned
\end{equation}

\item[(ii)] Let $1\leq q<2$, $1\leq r<2$. It holds that
\begin{equation}\label{ine3.24}
J_5\leq\frac{1}{t-s}\|v\|_{L^q(B_t\backslash B_\frac{3R}{2})}^{\frac{2q}{6-q}}\|\theta\|_{{L^r}(B_t\backslash B_\frac{3R}{2})}^{\frac{2r}{6-r}} \|v\|_{L^6(B_t\backslash B_\frac{3R}{2})}^{\frac{3(2-q)}{6-q}}\|\theta\|_{{L^6}(B_t\backslash B_\frac{3R}{2})}^\frac{3(2-r)}{6-r}.
\end{equation}

\item[(iii)] Let $1\leq q<2$, $2\leq r\leq 6$. For any $\varepsilon>0$, there exist a positive constant $C_\varepsilon$ such that
\begin{equation}\label{ine3.25}
\aligned
J_5\leq\varepsilon\|v\|_{L^6(B_t\backslash B_\frac{3R}{2})}^2+\frac{C_\varepsilon}{(t-s)^\frac{2(6-q)}{6+q}}R^{\frac{3(r-2)(6-q)}{r(6+q)}}\|v\|_{L^q(A_{R})}^{\frac{4q}{6+q}}\|\theta\|_{L^r(A_{R})}^\frac{2(6-q)}{6+q}.
\endaligned
\end{equation}

\item[(iv)] Let $1\leq q<2$, $2\leq r\leq6$. It holds that
\begin{equation}\label{ine3.26}
J_5\leq\frac{C}{t-s}R^{3(\frac{1}{2}-\frac{1}{r})}\|v\|_{L^q(B_t\backslash B_\frac{3R}{2})}^{\frac{2q}{6-q}}\|\theta\|_{{L^r}(B_t\backslash B_\frac{3R}{2})} \|v\|_{L^6(B_t\backslash B_\frac{3R}{2})}^{\frac{3(2-q)}{6-q}}.
\end{equation}

\item[(v)] Let $2\leq q\leq6$, $1\leq r<2$. For any $\varepsilon>0$, there exist a positive constant $C_\varepsilon$ such that
\begin{equation}\label{ine3.27}
\aligned
J_5\leq\varepsilon\|\theta\|_{L^6(B_t\backslash B_\frac{3R}{2})}^2+\frac{C_\varepsilon}{(t-s)^\frac{2(6-r)}{6+r}}R^{\frac{3(q-2)(6-r)}{q(6+r)}}\|v\|_{L^q(A_{R})}^{\frac{2(6-r)}{6+r}}\|\theta\|_{L^r(A_{R})}^\frac{4r}{6+r}.
\endaligned
\end{equation}

\item[(vi)] Let $2\leq q\leq6$, $1\leq r<2$. It holds that
\begin{equation}\label{ine3.28}
J_5\leq\frac{C}{t-s}R^{3(\frac{1}{2}-\frac{1}{q})}\|v\|_{L^q(B_t\backslash B_\frac{3R}{2})}\|\theta\|_{{L^r}(B_t\backslash B_\frac{3R}{2})}^{\frac{2r}{6-r}} \|\theta\|_{L^6(B_t\backslash B_\frac{3R}{2})}^{\frac{3(2-r)}{6-r}}.
\end{equation}

\item[(vii)] Let $2\leq q\leq6$, $2\leq r\leq 6$. It holds that
\begin{equation}\label{ine3.29}
\aligned
J_5\leq \frac{C}{t-s}R^{3(1-\frac{1}{q}-\frac{1}{r})}\|v\|_{L^q(A_{R})}\|\theta\|_{L^r(A_{R})}.
\endaligned
\end{equation}
\end{itemize}
\end{Lem}

\begin{proof}
By the H\"{o}lder inequality, we obtain
\begin{equation}\label{ine3.30}
J_5\leq \frac{1}{t-s}\|v\|_{L^2(B_t \backslash B_s)}\|\theta\|_{L^2(B_t \backslash B_s)}.
\end{equation}
When $1\leq q<2, 1\leq r<2$, substituting the following two interpolation inequalities
\begin{equation}\label{ine3.31}
\|v\|_{L^2(B_t \backslash B_s)}\leq\|v\|_{L^q(B_t \backslash B_s)}^{\frac{2q}{6-q}}\|v\|_{L^6(B_t \backslash B_s)}^{\frac{3(2-q)}{6-q}},
\end{equation}
\begin{equation}\label{ine3.32}
\|\theta\|_{L^2(B_t \backslash B_s)}\leq\|\theta\|_{L^r(B_t \backslash B_s)}^{\frac{2r}{6-r}}\|\theta\|_{L^6(B_t \backslash B_s)}^{\frac{3(2-r)}{6-r}}
\end{equation}
into \eqref{ine3.30}, we arrive at the estimate \eqref{ine3.24}.
Applying the Young inequality to \eqref{ine3.24} twice, we have
\begin{align*}
J_5&\leq\varepsilon\|v\|_{L^6(B_t\backslash B_\frac{3R}{2})}^2+C_\varepsilon\left(\frac{1}{t-s}\|v\|_{L^q(B_t\backslash B_\frac{3R}{2})}^{\frac{2q}{6-q}}\|\theta\|_{{L^r}(B_t\backslash B_\frac{3R}{2})}^{\frac{2r}{6-r}}\|\theta\|_{{L^6}(B_t\backslash B_\frac{3R}{2})}^\frac{3(2-r)}{6-r}\right)^\frac{12-2q}{6+q}\\
&\leq \varepsilon\|v\|_{L^6(B_t\backslash B_\frac{3R}{2})}^2+\varepsilon\|\theta\|_{L^6(B_t\backslash B_\frac{3R}{2})}^2+\frac{C_\varepsilon}{(t-s)^{\frac{(6-q)(6-r)}{q(6-r)+r(6-q)}}}\|v\|_{L^q(A_{R})}^{\frac{2q(6-r)}{q(6-r)+r(6-q)}}\|\theta\|_{L^r(A_{R})}^{\frac{2r(6-q)}{q(6-r)+r(6-q)}}.
\end{align*}
When $1\leq q<2$, $2\leq r\leq6$, combining \eqref{ine3.30} and \eqref{ine3.31} and then using the H\"{o}lder inequality, we get
\begin{align*}
J_5&\leq\frac{1}{t-s}\|v\|_{L^q(B_t\backslash B_\frac{3R}{2})}^{\frac{2q}{6-q}}\|\theta\|_{{L^2}(B_t\backslash B_\frac{3R}{2})} \|v\|_{L^6(B_t\backslash B_\frac{3R}{2})}^{\frac{3(2-q)}{6-q}}\\
&\leq\frac{C}{t-s}R^{3(\frac{1}{2}-\frac{1}{r})}\|v\|_{L^q(B_t\backslash B_\frac{3R}{2})}^{\frac{2q}{6-q}}\|\theta\|_{{L^r}(B_t\backslash B_\frac{3R}{2})} \|v\|_{L^6(B_t\backslash B_\frac{3R}{2})}^{\frac{3(2-q)}{6-q}}.
\end{align*}
Applying the Young inequality to \eqref{ine3.26}, we have
\begin{align*}
J_5\leq\varepsilon\|v\|_{L^6(B_t\backslash B_\frac{3R}{2})}^2+\frac{C_\varepsilon}{(t-s)^\frac{2(6-q)}{6+q}}R^{\frac{3(r-2)(6-q)}{r(6+q)}}\|v\|_{L^q(A_{R})}^{\frac{4q}{6+q}}\|\theta\|_{L^r(A_{R})}^\frac{2(6-q)}{6+q}.
\end{align*}
When $2\leq q\leq6, 1\leq r<2$, combining \eqref{ine3.30} and \eqref{ine3.32} and then using the H\"{o}lder inequality, we get \eqref{ine3.28}. Applying the Young inequality to \eqref{ine3.28}, we find that \eqref{ine3.27} holds.

For $2\leq q \leq6, 2\leq r \leq6$, applying the H\"{o}lder inequality to \eqref{ine3.30}, we obtain
\begin{align*}
J_5&\leq\frac{C}{t-s}R^{3(1-\frac{1}{q}-\frac{1}{r})}\|v\|_{L^q(B_t \backslash B_s)}\|\theta\|_{L^r(B_t \backslash B_s)}\\
&\leq\frac{C}{t-s}R^{3(1-\frac{1}{q}-\frac{1}{r})}\|v\|_{L^q(A_{R})}\|\theta\|_{L^r(A_{R})}.
\end{align*}

\end{proof}

With the above preparations, we proceed to prove Theorem \ref{main1}.
\begin{proof}[{\bf Proof of Theorem \ref{main1}}]
{\bf Assume that $\mathrm{(A1)}$ holds.} Then we can choose a sequence $R_j\nearrow+\infty$ such that
\begin{equation*}
\lim\limits_{j\rightarrow+\infty}X_{p,\alpha}(R_j)=0,\;\lim\limits_{j\rightarrow+\infty}Y_{q,\beta}(R_j)<+\infty,\;\lim\limits_{j\rightarrow+\infty}Z_{r,\gamma}(R_j)<+\infty.
\end{equation*}
We only consider the case $p,q,r\in[3,\frac{9}{2}]\times[1,2)\times[1,2)$ as the remaining cases can be treated similarly. Let $f$ be the function defined by \eqref{ine3.1}.

Combining  \eqref{ine3.3}, \eqref{ine3.12}, \eqref{ine3.14},  \eqref{ine3.17}, \eqref{ine3.20} and \eqref{ine3.23}, we derive that
\begin{align*}
f(s)\leq&\frac{1}{2}f(t)+\frac{C}{(t-s)^{\frac{6}{p}-1}}\|u\|_{L^p(A_{R})}^{2}+\frac{C}{(t-s)^2}R^{3-\frac{6}{p}}\|u\|_{L^{p}(A_{R})}^{2}+\frac{C}{(t-s)^{\frac{6}{q}-1}}\|v\|_{L^{q}(A_{R})}^{2}\\
&+\frac{C}{(t-s)^2}R^{3-\frac{6}{q}}\|v\|_{L^q(A_{R})}^{2}+\frac{C}{(t-s)^{\frac{6}{r}-1}}\|\theta\|_{L^{r}(A_{R})}^{2}+\frac{C}{(t-s)^2}R^{3-\frac{6}{r}}\|\theta\|_{L^r(A_{R})}^{2}\\
&+\frac{C}{(t-s)}R^{3-\frac{9}{p}}\|u\|_{L^p(A_{R})}^3+\frac{C}{(t-s)^\frac{(6-q)p'}{(3-p')q}}\|u\|_{L^p(A_{R})}^\frac{(6-q)p'}{(3-p')q}\|v\|_{L^q(A_{R})}^2\\
&+\frac{C}{(t-s)^\frac{(6-r)p'}{(3-p')r}}\|u\|_{L^p(A_{R})}^\frac{(6-r)p'}{(3-p')r}\|\theta\|_{L^r(A_{R})}^2\\
&+\frac{C}{(t-s)^{\frac{(6-q)(6-r)}{q(6-r)+r(6-q)}}}\|v\|_{L^q(A_{R})}^{\frac{2q(6-r)}{q(6-r)+r(6-q)}}\|\theta\|_{L^r(A_{R})}^{\frac{2r(6-q)}{q(6-r)+r(6-q)}}.
\end{align*}
Applying Lemma \ref{Lem2.3} to the above function inequality, and taking $s=\sqrt{3}R$ and $t=2R$, we conclude that
\begin{align*}
f(R)\leq&f\left(\sqrt{3}R\right)\leq CR^{1-\frac{6}{p}}\|u\|_{L^p(A_{R})}^2+CR^{1-\frac{6}{q}}\|v\|_{L^{q}(A_{R})}^{2}+CR^{1-\frac{6}{r}}\|\theta\|_{L^{r}(A_{R})}^{2}\\
&+CR^{2-\frac{9}{p}}\|u\|_{L^p(A_{R})}^3+CR^{-\frac{(6-q)p'}{(3-p')q}}\|u\|_{L^p(A_{R})}^\frac{(6-q)p'}{(3-p')q}\|v\|_{L^q(A_{R})}^2\\
&+CR^{-\frac{(6-r)p'}{(3-p')r}}\|u\|_{L^p(A_{R})}^\frac{(6-r)p'}{(3-p')r}\|\theta\|_{L^r(A_{R})}^2\\
&+CR^{-\frac{(6-q)(6-r)}{q(6-r)+r(6-q)}}\|v\|_{L^q(A_{R})}^{\frac{2q(6-r)}{q(6-r)+r(6-q)}}\|\theta\|_{L^r(A_{R})}^{\frac{2r(6-q)}{q(6-r)+r(6-q)}}.
\end{align*}
Hence, it holds that
\begin{align*}
f(R)\leq&CR^{1-\frac{6}{p}+2\alpha}X^2_{p,\alpha}(R)+CR^{1-\frac{6}{q}+2\beta}Y^2_{q,\beta}(R)+CR^{1-\frac{6}{r}+2\gamma}Z^2_{r,\gamma}(R)\\
&+CR^{2-\frac{9}{p}+3\alpha}X^3_{p,\alpha}(R)+CR^{-\frac{(6-q)p'}{(3-p')q}+\frac{(6-q)p'}{(3-p')q}\alpha+2\beta}X^\frac{(6-q)p'}{(3-p')q}_{p,\alpha}(R)Y^2_{q,\beta}(R)\\
&+CR^{-\frac{(6-r)p'}{(3-p')r}+\frac{(6-r)p'}{(3-p')r}\alpha+2\gamma}X^\frac{(6-r)p'}{(3-p')r}_{p,\alpha}(R)Z^2_{r,\gamma}(R)\\
&+CR^\frac{-(6-q)(6-r)+2q(6-r)\beta+2r(6-q)\gamma}{q(6-r)+r(6-q)}Y^{\frac{2q(6-r)}{q(6-r)+r(6-q)}}_{q,\beta}(R)Z^{\frac{2r(6-q)}{q(6-r)+r(6-q)}}_{r,\gamma}(R).
\end{align*}
Letting $R=R_j\rightarrow+\infty$, thanks to
$$1-\frac{6}{p}+2\alpha<0,\;1-\frac{6}{q}+2\beta\leq0,\;1-\frac{6}{r}+2\gamma\leq0,\;2-\frac{9}{p}+3\alpha\leq0,$$
and
\begin{align*}
&-\frac{(6-q)p'}{(3-p')q}+\frac{(6-q)p'}{(3-p')q}\alpha+2\beta=\frac{(6-q)p}{(2p-3)q}\left(-1+\alpha+\frac{(4p-6)q}{(6-q)p}\beta\right)\leq0,\\
&-\frac{(6-r)p'}{(3-p')r}+\frac{(6-r)p'}{(3-p')r}\alpha+2\gamma=\frac{(6-r)p}{(2p-3)r}\left(-1+\alpha+\frac{(4p-6)r}{(6-r)p}\gamma\right)\leq0,
\end{align*}
\begin{align*}
&\frac{-(6-q)(6-r)+2q(6-r)\beta+2r(6-q)\gamma}{q(6-r)+r(6-q)}\\
=&\frac{(6-q)(6-r)}{q(6-r)+r(6-q)}\left(-1+\frac{2q}{6-q}\beta+\frac{2r}{6-r}\gamma\right)\leq0,
\end{align*}
we get $u,v,\theta\in L^6(\mathbb{R}^3)$ and $\nabla u,\nabla v,\nabla \theta\in L^2(\mathbb{R}^3)$. Furthermore, it holds that
\begin{equation}\label{ine3.33}
\aligned
&\lim_{R\rightarrow+\infty}\|u\|_{L^6(A_{R})}=\lim_{R\rightarrow+\infty}\|v\|_{L^6(A_{R})}=\lim_{R\rightarrow+\infty}\|\theta\|_{L^6(A_{R})}=0,\\
&\lim_{R\rightarrow+\infty}\|\nabla u\|_{L^2(A_{R})}=\lim_{R\rightarrow+\infty}\|\nabla v\|_{L^2(A_{R})}=\lim_{R\rightarrow+\infty}\|\nabla \theta\|_{L^2(A_{R})}=0.
\endaligned
\end{equation}

Combining  \eqref{ine3.3}, \eqref{ine3.13}, \eqref{ine3.14}, \eqref{ine3.17}, \eqref{ine3.20} and \eqref{ine3.24}, we have
\begin{align*}
f(s)&\leq\frac{1}{2}\|\nabla u\|_{L^2(B_t\backslash B_\frac{3R}{2})}^2+\frac{CR^2}{(t-s)^2}\left(\|u\|_{L^6(A_{R})}^2 +\|v\|_{L^6(A_{R})}^2 +\|\theta\|_{L^6(A_{R})}^2\right)\\
&+\frac{C}{t-s}R^{3-\frac{9}{p}}\|u\|_{L^{p}(A_{R})}^{3}+\frac{1}{2}\|v\|_{L^6(B_t\backslash B_\frac{3R}{2})}^2+\frac{C}{(t-s)^\frac{(6-q)p'}{(3-p')q}}\|u\|_{L^p(A_{R})}^\frac{(6-q)p'}{(3-p')q}\|v\|_{L^q(A_{R})}^2\\
&+\frac{1}{2}\|\theta\|_{L^6(B_t\backslash B_\frac{3R}{2})}^2+\frac{C}{(t-s)^\frac{(6-r)p'}{(3-p')r}}\|u\|_{L^p(A_{R})}^\frac{(6-r)p'}{(3-p')r}\|\theta\|_{L^r(A_{R})}^2\\
&+\frac{C}{t-s}\|v\|_{L^q(B_t\backslash B_\frac{3R}{2})}^{\frac{2q}{6-q}}\|\theta\|_{{L^r}(B_t\backslash B_\frac{3R}{2})}^{\frac{2r}{6-r}} \|v\|_{L^6(B_t\backslash B_\frac{3R}{2})}^{\frac{3(2-q)}{6-q}}\|\theta\|_{{L^6}(B_t\backslash B_\frac{3R}{2})}^\frac{3(2-r)}{6-r}.
\end{align*}
Taking $s=\sqrt{3}R$ and $t=2R$ in the above inequality, we get
\begin{align*}
f(R)\leq& f\left(\sqrt{3}R\right)\leq \frac{1}{2}\|\nabla u\|_{L^{2}(A_{R})}^2+C\left(\|u\|_{L^6(A_{R})}^2+\|v\|_{L^6(A_{R})}^2+\|\theta\|_{L^6(A_{R})}^2\right)\\
&+CR^{2-\frac{9}{p}}\|u\|_{L^p(A_{R})}^3+CR^{-\frac{(6-q)p'}{(3-p')q}}\|u\|_{L^p(A_{R})}^\frac{(6-q)p'}{(3-p')q}\|v\|_{L^q(A_{R})}^2\\
&+CR^{-\frac{(6-r)p'}{(3-p')r}}\|u\|_{L^p(A_{R})}^\frac{(6-r)p'}{(3-p')r}\|\theta\|_{L^r(A_{R})}^2\\
&+CR^{-1}\|v\|_{L^q(A_{R})}^{\frac{2q}{6-q}}\|\theta\|_{{L^r}(A_{R})}^{\frac{2r}{6-r}} \|v\|_{L^6(A_{R})}^{\frac{3(2-q)}{6-q}}\|\theta\|_{{L^6}(A_{R})}^\frac{3(2-r)}{6-r}.
\end{align*}
Therefore, we have
\begin{align*}
f(R)\leq&\frac{1}{2}\|\nabla u\|_{L^{2}(A_{R})}^2+C\left(\|u\|_{L^6(A_{R})}^2+\|v\|_{L^6(A_{R})}^2+\|\theta\|_{L^6(A_{R})}^2\right)\\
&+CR^{2-\frac{9}{p}+3\alpha}X^3_{p,\alpha}(R)+CR^{-\frac{(6-q)p'}{(3-p')q}+\frac{(6-q)p'}{(3-p')q}\alpha+2\beta}X^\frac{(6-q)p'}{(3-p')q}_{p,\alpha}(R)Y^2_{q,\beta}(R)\\
&+CR^{-\frac{(6-r)p'}{(3-p')r}+\frac{(6-r)p'}{(3-p')r}\alpha+2\gamma}X^\frac{(6-r)p'}{(3-p')r}_{p,\alpha}(R)Z^2_{r,\gamma}(R)\\
&+CR^{-1+\frac{2q}{6-q}\beta+\frac{2r}{6-r}\gamma}Y^{\frac{2q}{6-q}}_{q,\beta}(R)Z^{\frac{2r}{6-r}}_{r,\gamma}(R) \|v\|_{L^6(A_{R})}^{\frac{3(2-q)}{6-q}}\|\theta\|_{{L^6}(A_{R})}^\frac{3(2-r)}{6-r}.
\end{align*}
Letting $R=R_j\rightarrow+\infty$ and thanks to \eqref{ine3.33}, we obtain that $u=v=0$ and $\theta=0$.

{\bf Assume that $\mathrm{(A2)}$ holds.} Then there exists a sequence $R_j\nearrow+\infty$ such that
$$
\lim\limits_{j\rightarrow+\infty}X_{\alpha,p,R_j}<+\infty,\;\lim\limits_{j\rightarrow+\infty}Y_{\beta,q,R_j}<+\infty,\;\lim\limits_{j\rightarrow+\infty}Z_{\gamma,r,R_j}<+\infty.
$$
We only consider the case $p,q,r\in(\frac{3}{2},3)\times[1,2)\times[1,2)$ as the remaining cases can be treated similarly.

Combining  \eqref{ine3.3}, \eqref{ine3.12}, \eqref{ine3.16}, \eqref{ine3.17}, \eqref{ine3.20} and \eqref{ine3.23}, we derive that
\begin{align*}
f(s)\leq&\frac{1}{2}f(t)+\frac{C}{(t-s)^{\frac{6}{p}-1}}\|u\|_{L^p(A_{R})}^{2}+\frac{C}{(t-s)^2}R^{3-\frac{6}{p}}\|u\|_{L^{p}(A_{R})}^{2}+\frac{C}{(t-s)^{\frac{6}{q}-1}}\|v\|_{L^{q}(A_{R})}^{2}\\
&+\frac{C}{(t-s)^2}R^{3-\frac{6}{q}}\|v\|_{L^q(A_{R})}^{2}+\frac{C}{(t-s)^{\frac{6}{r}-1}}\|\theta\|_{L^{r}(A_{R})}^{2}+\frac{C}{(t-s)^2}R^{3-\frac{6}{r}}\|\theta\|_{L^r(A_{R})}^{2}\\
&+\frac{C}{(t-s)^\frac{6-p}{2p-3}}\|u\|_{L^p(A_{R})}^\frac{3p}{2p-3}+\frac{C}{(t-s)^\frac{(6-q)p'}{(3-p')q}}\|u\|_{L^p(A_{R})}^\frac{(6-q)p'}{(3-p')q}\|v\|_{L^q(A_{R})}^2\\
&+\frac{C}{(t-s)^\frac{(6-r)p'}{(3-p')r}}\|u\|_{L^p(A_{R})}^\frac{(6-r)p'}{(3-p')r}\|\theta\|_{L^r(A_{R})}^2\\
&+\frac{C}{(t-s)^{\frac{(6-q)(6-r)}{q(6-r)+r(6-q)}}}\|v\|_{L^q(A_{R})}^{\frac{2q(6-r)}{q(6-r)+r(6-q)}}\|\theta\|_{L^r(A_{R})}^{\frac{2r(6-q)}{q(6-r)+r(6-q)}}.
\end{align*}
Applying Lemma \ref{Lem2.3} to the above function inequality, we obtain
\begin{align}\label{add1}
f(R)\leq&CR^{1-\frac{6}{p}}\|u\|_{L^p(A_{R})}^2+CR^{1-\frac{6}{q}}\|v\|_{L^{q}(A_{R})}^{2}+CR^{1-\frac{6}{r}}\|\theta\|_{L^{r}(A_{R})}^{2}+CR^\frac{p-6}{2p-3}\|u\|_{L^p(A_{R})}^\frac{3p}{2p-3}\notag\\
&+CR^{-\frac{(6-q)p'}{(3-p')q}}\|u\|_{L^p(A_{R})}^\frac{(6-q)p'}{(3-p')q}\|v\|_{L^q(A_{R})}^2+CR^{-\frac{(6-r)p'}{(3-p')r}}\|u\|_{L^p(A_{R})}^\frac{(6-r)p'}{(3-p')r}\|\theta\|_{L^r(A_{R})}^2\\
&+CR^{-\frac{(6-q)(6-r)}{q(6-r)+r(6-q)}}\|v\|_{L^q(A_{R})}^{\frac{2q(6-r)}{q(6-r)+r(6-q)}}\|\theta\|_{L^r(A_{R})}^{\frac{2r(6-q)}{q(6-r)+r(6-q)}}.\notag
\end{align}
Hence, we deduce
\begin{align*}
f(R)\leq&CR^{1-\frac{6}{p}+2\alpha}X^2_{p,\alpha}(R)+CR^{1-\frac{6}{q}+2\beta}Y^2_{q,\beta}(R)+CR^{1-\frac{6}{r}+2\gamma}Z^2_{r,\gamma}(R)\\
&+CR^{\frac{p-6}{2p-3}+\frac{3p}{2p-3}\alpha}X^\frac{3p}{2p-3}_{p,\alpha}(R)+CR^{-\frac{(6-q)p'}{(3-p')q}+\frac{(6-q)p'}{(3-p')q}\alpha+2\beta}X^\frac{(6-q)p'}{(3-p')q}_{p,\alpha}(R)Y^2_{q,\beta}(R)\\
&+CR^{-\frac{(6-r)p'}{(3-p')r}+\frac{(6-r)p'}{(3-p')r}\alpha+2\gamma}X^\frac{(6-r)p'}{(3-p')r}_{p,\alpha}(R)Z^2_{r,\gamma}(R)\\
&+CR^\frac{-(6-q)(6-r)+2q(6-r)\beta+2r(6-q)\gamma}{q(6-r)+r(6-q)}Y^{\frac{2q(6-r)}{q(6-r)+r(6-q)}}_{q,\beta}(R)Z^{\frac{2r(6-q)}{q(6-r)+r(6-q)}}_{r,\gamma}(R).
\end{align*}
Letting $R=R_j\rightarrow+\infty$, we get $u,v,\theta\in L^6(\mathbb{R}^3)$ and $\nabla u,\nabla v,\nabla \theta\in L^2(\mathbb{R}^3)$. Furthermore, \eqref{ine3.33} holds.

Combining  \eqref{ine3.3}, \eqref{ine3.13}, \eqref{ine3.15}, \eqref{ine3.18}, \eqref{ine3.21} and \eqref{ine3.24}, we have
\begin{align*}
f(R)\leq&\frac{1}{2}\|\nabla u\|_{L^{2}(A_{R})}^2+C\left(\|u\|_{L^6(A_{R})}^2 +\|v\|_{L^6(A_{R})}^2+\|\theta\|_{L^6(A_{R})}^2\right)\\
&+CR^{-1}\|u\|_{L^p(A_{R})}^{\frac{3p}{6-p}}\|u\|_{L^6(A_{R})}^{\frac{18-6p}{6-p}}+CR^{-1}\|u\|_{L^p(A_{R})}\|v\|_{L^q(A_{R})}^\frac{2(3-p')q}{(6-q)p'}\|v\|_{L^6(A_{R})}^\frac{12p'-6q}{(6-q)p'}\\
&+CR^{-1}\|u\|_{L^p(A_{R})}\|\theta\|_{L^r(A_{R})}^\frac{2(3-p')r}{(6-r)p'}\|\theta\|_{L^6(A_{R})}^\frac{12p'-6r}{(6-r)p'}\\
&+CR^{-1}\|v\|_{L^q(B_t\backslash B_\frac{3R}{2})}^{\frac{2q}{6-q}}\|\theta\|_{{L^r}(B_t\backslash B_\frac{3R}{2})}^{\frac{2r}{6-r}} \|v\|_{L^6(B_t\backslash B_\frac{3R}{2})}^{\frac{3(2-q)}{6-q}}\|\theta\|_{{L^6}(B_t\backslash B_\frac{3R}{2})}^\frac{3(2-r)}{6-r}.
\end{align*}
Consequently, we derive
\begin{align*}
f(R)\leq&\frac{1}{2}\|\nabla u\|_{L^{2}(A_{R})}^2+C\left(\|u\|_{L^6(A_{R})}^2 +\|v\|_{L^6(A_{R})}^2+\|\theta\|_{L^6(A_{R})}^2\right)\\
&+CR^{\frac{3p}{6-p}\alpha-1}X^{\frac{3p}{6-p}}_{p,\alpha}(R)\|u\|_{L^6(A_{R})}^{\frac{18-6p}{6-p}}\\
&+CR^{-1+\alpha+\frac{2(3-p')q}{(6-q)p'}\beta}X_{p,\alpha}(R)Y^\frac{2(3-p')q}{(6-q)p'}_{q,\beta}(R)\|v\|_{L^6(A_{R})}^\frac{12p'-6q}{(6-q)p'}\\
&+CR^{-1+\alpha+\frac{2(3-p')r}{(6-r)p'}\gamma}X_{p,\alpha}(R)Z^\frac{2(3-p')r}{(6-r)p'}_{r,\gamma}(R)\|\theta\|_{L^6(A_{R})}^\frac{12p'-6r}{(6-r)p'}\\
&+CR^{-1+\frac{2q}{6-q}\beta+\frac{2r}{6-r}\gamma}Y^{\frac{2q}{6-q}}_{q,\beta}(R)Z^{\frac{2r}{6-r}} _{r,\gamma}(R) \|v\|_{L^6(A_{R})}^{\frac{3(2-q)}{6-q}}\|\theta\|_{{L^6}(A_{R})}^\frac{3(2-r)}{6-r}.
\end{align*}
Letting $R=R_j\rightarrow+\infty$ and using \eqref{ine3.33}, we conclude that $u=v=0$ and $\theta=0$.
\end{proof}

\begin{Rem}\label{Rem3.1}
It should be noted that there are some strict inequalities in our assumptions of {\rm Theorem \ref{main1}}. Actually, the strict inequalities can be weakened in some special circumstances. We give some concrete remarks on those endpoint cases.
\begin{itemize}
\item[(i)] When $(p,q)\in(\frac{3}{2},3)\times[2p',6]$, the inequality in {\rm Assumption \ref{a1.2}(i)} becomes
\begin{equation*}
\alpha+2\beta<\frac{3}{p}+\frac{6}{q}-2,
\end{equation*}
which can be replaced by the equality
$$
\alpha+2\beta=\frac{3}{p}+\frac{6}{q}-2,
$$
but the price is that we need to assume in addition that
$$\liminf\limits_{R\rightarrow+\infty}X_{p,\alpha}(R)=0,\;\limsup\limits_{R\rightarrow+\infty}[Y_{q,\beta}(R)+Z_{r,\gamma}(R)]<+\infty,\text{ or }$$
$$\limsup\limits_{R\rightarrow+\infty}[X_{p,\alpha}(R)+Z_{r,\gamma}(R)]<+\infty,\;\liminf\limits_{R\rightarrow+\infty}Y_{q,\beta}(R)=0.$$
Indeed, when $(p,q)\in(\frac{3}{2},3)\times[2p',6]$, we need to deal with the term
$$CR^{2-\frac{3}{p}-\frac{6}{q}+\alpha+2\beta}X_{p,\alpha}(R)Y^2_{q,\beta}(R)$$
and require this term to tend to zero.
\item[(ii)]
Similarly to $\mathrm{(i)}$, when $(p,r)\in(\frac{3}{2},3)\times[2p',6]$, the inequality in {\rm Assumption \ref{a1.2}(ii)} becomes
\begin{equation*}
\alpha+2\gamma<\frac{3}{p}+\frac{6}{r}-2,
\end{equation*}
which can be replaced by the equality
$$
\alpha+2\gamma=\frac{3}{p}+\frac{6}{r}-2,
$$
but the price is that we need to assume in addition that
$$\liminf\limits_{R\rightarrow+\infty}X_{p,\alpha}(R)=0,\;\limsup\limits_{R\rightarrow+\infty}[Y_{q,\beta}(R)+Z_{r,\gamma}(R)]<+\infty,\text{ or }$$
$$\limsup\limits_{R\rightarrow+\infty}[X_{p,\alpha}(R)+Y_{q,\beta}(R)]<+\infty,\;\liminf\limits_{R\rightarrow+\infty}Z_{r,\gamma}(R)=0.$$
\item[(iii)]
When $(q,r)\in[2,6]\times[2,6]$, the inequality in {\rm Assumption \ref{a1.2}(iii)} becomes
\begin{equation*}
\beta+\gamma<\frac{3}{q}+\frac{3}{r}-2,
\end{equation*}
which can be replaced by the equality
$$\beta+\gamma=\frac{3}{q}+\frac{3}{r}-2,$$
but the price is that we need to assume in addition that
\begin{itemize}
\item[(1)] when $p\in[3,\frac{9}{2}]$, we assume
$$\liminf\limits_{R\rightarrow+\infty}Y_{q,\beta}(R)=0\;\text{ or }\liminf\limits_{R\rightarrow+\infty}Z_{r,\gamma}(R)=0.$$
\item[(2)] when $p\in(3,\frac{3}{2})$, we assume
$$\liminf\limits_{R\rightarrow+\infty}Y_{q,\beta}(R)=0,\;\limsup\limits_{R\rightarrow+\infty}[X_{p,\alpha}(R)+Z_{r,\gamma}(R)]<+\infty,\text{ or }$$
$$\limsup\limits_{R\rightarrow+\infty}[X_{p,\alpha}(R)+Y_{q,\beta}(R)]<+\infty,\;\liminf\limits_{R\rightarrow+\infty}Z_{r,\gamma}(R)=0.$$
\end{itemize}
Indeed, when $(q,r)\in[2,6]\times[2,6]$, we need to deal with the term
$$CR^{2-\frac{3}{q}-\frac{3}{r}+\beta+\gamma}Y_{q,\beta}(R)Z_{r,\gamma}(R)$$
and require this term to tend to zero.
\item[(iv)] When $(p,q)\in(\frac{3}{2},3)\times[2p',6]$, {\rm Assumption \ref{a1.2}(i)}
can be replaced by $\alpha=\frac{3}{p}-1$ and  $\beta=\frac{3}{q}-\frac{1}{2}$. Indeed, using the H\"{o}lder inequality, we have
$$CR^{2-\frac{3}{p}-\frac{6}{q}+\alpha+2\beta}X_{p,\alpha}(R)Y^2_{q,\beta}(R)\leq CX_{p,\alpha}(R)\|v\|_{L^6(A_R)}^2.$$
\item[(v)] Similarly to $\mathrm{(iv)}$, when $(p,r)\in(\frac{3}{2},3)\times[2p',6]$, {\rm Assumption \ref{a1.2}(ii)}
can be replaced by $\alpha=\frac{3}{p}-1$ and $\gamma=\frac{3}{r}-\frac{1}{2}$.
\item[(vi)] When $(q,r)\in[2,6]\times[2,6]$, {\rm Assumption \ref{a1.2}(iii)} can be replaced by
$(\beta,\gamma)=(0,\frac{3}{r}-\frac{1}{2})$ with $q=2$, or $(\beta,\gamma)=(\frac{3}{q}-\frac{1}{2},0)$ with $r=2$.
Indeed, using the H\"{o}lder inequality, we have
$$CR^{2-\frac{3}{q}-\frac{3}{r}+\beta+\gamma}Y_{q,\beta}(R)Z_{r,\gamma}(R)\leq CY_{q,\beta}(R)\|\theta\|_{L^6(A_R)},\text{ with }q=2,$$
$$CR^{2-\frac{3}{q}-\frac{3}{r}+\beta+\gamma}Y_{q,\beta}(R)Z_{r,\gamma}(R)\leq C\|v\|_{L^6(A_R)}Z_{r,\gamma}(R),\text{ with }r=2.$$
\end{itemize}
\end{Rem}

\begin{Rem}\label{Rem3.2}
Since the range of $p$ is divided into two cases: $\frac{3}{2}<p<3$ and $3\leq p\leq\frac{9}{2}$, the range of $q$ is divided into three cases: $1\leq q<2$, $2\leq q<2p'$ and $2p'\leq q\leq6$, the range of $r$ is also divided into three cases: $1\leq r<2$, $2\leq r<2p'$ and $2p'\leq r\leq6$,
by the multiplication principle, there are $18$ cases in {\rm Theorem \ref{main1}}. However, we mention that the case of $\frac{3}{2}<p<3$, $2p'\leq q\leq6$, $2p'\leq r\leq6$
would not happen, since we can derive a contradiction
$$
0\leq\beta+\gamma<\frac{3}{q}+\frac{3}{r}-2\leq \frac{3}{2p'}+\frac{3}{2p'}-2=1-\frac{3}{p}<0.
$$
Therefore, {\rm Theorem \ref{main1}} includes $17$ different cases.
\end{Rem}

\begin{proof}[{\bf Proof of Corollary \ref{Cor1.1}}]
Since $u\in L^p(\mathbb{R}^3)$, $v\in L^q(\mathbb{R}^3)$, $\theta\in L^r(\mathbb{R}^3)$, we obtain
$$\lim_{R\rightarrow+\infty}\|u\|_{L^p\left(A_{R}\right)}=\lim_{R\rightarrow+\infty}\|v\|_{L^q\left(A_{R}\right)}=\lim_{R\rightarrow+\infty}\|\theta\|_{L^r\left(A_{R}\right)}=0.$$
Applying Theorem \ref{main1} with $\alpha=\beta=\gamma=0$ and observing the endpoint cases in Remark \ref{Rem3.1}(i) (ii) (iii), we obtain the conclusions.
\end{proof}

\begin{proof}[{\bf Proof of Corollary \ref{Cor1.2}}]
We only prove the case when $\mathrm{(C1)}$ holds. Since $\mathrm{(C1)}$ indicates
$$
\lim_{R\rightarrow+\infty}\|u\|_{L^\frac{9}{2}\left(A_{R}\right)}=\lim_{R\rightarrow+\infty}\|\theta\|_{L^\frac{9}{4}\left(A_{R}\right)}=0,\;\limsup_{R\rightarrow+\infty}\|v\|_{L^\frac{9}{2}\left(A_{R}\right)}<+\infty,
$$
taking $\alpha=\beta=\gamma=0$, $p=q=\frac{9}{2}$, $r=\frac{9}{4}$ in Theorem \ref{main1} and considering the endpoint cases in Remark \ref{Rem3.1}(i) (iii), we obtain the desired result.
\end{proof}

\section{Proof of Theorem \ref{main2}}\label{sec4}

In this section, let $\eta$ be a cut-off function defined by
\begin{align*}
	\eta(x)= \begin{cases}
		1, & |x| <\frac{3R}{2}, \\
		4-\frac{2}{R}|x|,& \frac{3R}{2}\leq |x|\leq 2R,\\
		0, & |x| >2R.
	\end{cases}
\end{align*}
For any $R>0$, we define the energy function $E(R)$ by
\begin{align}\label{ine4.1}
	\aligned
	E(R)=\int_{\mathbb{R}^3}\left(|\nabla u|^{2}+|\nabla v|^{2}+|\nabla \theta|^{2}\right)\eta(x)d x.
	\endaligned
\end{align}
We will show some properties of $E(R)$ in the next two lemmas. In Lemma \ref{Lem4.1}, we establish a lower bound estimate for the derivative $E'(R)$.
In Lemma \ref{Lem4.2}, we establish an upper bound estimate for $E(R)$.
\begin{Lem}\label{Lem4.1}
	Let $(u,\pi,v,\theta)$ be a smooth solution of \eqref{equ1.1} and $E(R)$ be defined by \eqref{ine4.1}. Then we have
	\begin{align}\label{ine4.2}
		E'(R)\geq\frac{3}{R}\int_{A_R}\left(|\nabla u|^{2}+|\nabla v|^{2}+|\nabla \theta|^{2}\right)dx.
	\end{align}
Consequently, $E(R)$ is a non-decreasing function with respect to $R$.
\end{Lem}
\begin{proof}
	We rewrite $E(R)$ as the following form
	$$
	E(R)=\int_{B_{\frac{3R}{2}}}\left(|\nabla u|^{2}+|\nabla v|^{2}+|\nabla \theta|^{2}\right) dx+\int_{A_R}\left(|\nabla u|^{2}+|\nabla v|^{2}+|\nabla \theta|^{2}\right)\left(-\frac{2}{R}|x|+4\right)dx.
	$$
	By a direct calculation, we obtain
	\begin{align*}
		E'(R)=&\frac{3}{2}\int_{\partial B_{\frac{3R}{2}}}\left(|\nabla u|^{2}+|\nabla v|^{2}+|\nabla \theta|^{2}\right) dS+\int_{A_R}\left(|\nabla u|^{2}+|\nabla v|^{2}+|\nabla \theta|^{2}\right)\frac{2}{R^2}|x|dx\\
		&+2\int_{\partial B_{2R}}\left(|\nabla u|^{2}+|\nabla v|^{2}+|\nabla \theta|^{2}\right)\left(-\frac{2}{R}\cdot2R+4\right)dS\\
		&-\frac{3}{2}\int_{\partial B_{\frac{3R}{2}}}\left(|\nabla u|^{2}+|\nabla v|^{2}+|\nabla \theta|^{2}\right)\left(-\frac{2}{R}\cdot\frac{3R}{2}+4\right)dS\\
		=&\int_{A_R}\left(|\nabla u|^{2}+|\nabla v|^{2}+|\nabla \theta|^{2}\right)\frac{2}{R^2}|x|dx\\
		\geq&\frac{3}{R}\int_{A_R}\left(|\nabla u|^{2}+|\nabla v|^{2}+|\nabla \theta|^{2}\right)dx.
	\end{align*}
\end{proof}

We use $\overline{\varphi}_R$ to represent the mean value of $\varphi$ on the annulus $A_R$. Denote
\begin{align}\label{eq4.3}
U=u-\overline{u}_R,\;V=v-\overline{v}_R,\;\Theta=\theta-\overline{\theta}_R.
\end{align}
\begin{Lem}\label{Lem4.2}
	Let $(u,\pi,v,\theta)$ be a smooth solution of \eqref{equ1.1} and $E(R)$ be defined by \eqref{ine4.1}. Then for any $R>0$, it holds that
	\begin{align}\label{ine4.3}
		E(R)\leq& C\left(\|\nabla u\|_{L^2(A_R)}^2+\|\nabla v\|_{L^2(A_R)}^2+\|\nabla \theta\|_{L^2(A_R)}^2\right)+CR^{\frac{1}{2}-\frac{3}{q}}\|v\|_{L^q(A_R)}\|\nabla v\|_{L^2(A_R)}\notag\\
		&+CR^{-1}\|u\|_{L^p(A_R)}\|U\|_{L^{2p'}(A_R)}^2+CR^{-1}\|u\|_{L^p(A_R)}\|V\|_{L^{2p'}(A_R)}^2\notag\\
		&+CR^{-1}\|u\|_{L^p(A_R)}\|\Theta\|_{L^{2p'}(A_R)}^2+CR^{2-\frac{3}{p}-\frac{6}{q}}\|u\|_{L^p(A_R)}\|v\|_{L^q(A_R)}^{2}\\
		&+CR^{-1}\|V\|_{L^2(A_R)}\|\Theta\|_{L^2(A_R)}+CR^{2-\frac{3}{q}-\frac{3}{r}}\|v\|_{L^{q}(A_R)}\|\theta\|_{L^r(A_R)},\notag
	\end{align}
where $U$, $V$, $\Theta$ are defined by \eqref{eq4.3}.
\end{Lem}
\begin{proof}
	By Lemma \ref{Lem2.1}, there exists $w\in W_{0}^{1,\sigma}(A_R)$ such that $w$ satisfies the following equation
	\begin{align*}
		\mathrm{div} w=U\cdot\nabla\eta \text{ in }A_R,
	\end{align*}
	with the estimate
	\begin{align}\label{ine4.4}
		\|\nabla w\|_{L^\sigma(A_R)}\leq C\|U\cdot\nabla\eta\|_{L^\sigma(A_R)}\leq CR^{-1}\|U\|_{L^\sigma(A_R)},
	\end{align}
	for any $1<\sigma<+\infty$. We extend $w$ by zero to $B_\frac{3R}{2}$, then $w\in W_{0}^{1,\sigma}(B_{2R}).$
	
	Obviously, $(U,\pi,v,\Theta)$ satisfies
	\begin{align}\label{equ4.5}
		\left\{
		\begin{array}{ll}
			-\Delta U+(u\cdot\nabla) U+\nabla \pi+\divg(v\otimes v)=0,  \\
			-\Delta v+(u\cdot\nabla) v+\nabla \Theta+(v\cdot\nabla)U=0, \\
			-\Delta \Theta+u\cdot\nabla\Theta+\divg v=0,\\
			\divg U=0.
		\end{array}
		\right.
	\end{align}
Denote the $i$-th components of $U$ and $v$ by $U_i$ and $v_i$, respectively. Multiply both sides of $\eqref{equ4.5}_{1}$, $\eqref{equ4.5}_{2}$ and $\eqref{equ4.5}_{3}$ by $U \eta-w$, $v \eta$ and $\Theta \eta$ respectively, integrate over $B_{2R}$ and apply integration by parts. This procedure yields
	\begin{align}\label{ine4.6}
		E&(R)=\int_{B_{2R}}\left(|\nabla U|^{2}\eta+|\nabla v|^{2}\eta+|\nabla \Theta|^{2}\eta\right) d x\notag\\
		=&-\int_{B_{2R}}\Big[\left(\nabla \eta \cdot \nabla\right)U\cdot U+\left(\nabla \eta \cdot \nabla\right)v\cdot v+\nabla \Theta \cdot \left(\Theta\nabla \eta\right)\Big]d x+\int_{B_{2R}}\nabla U:\nabla w dx\\
		&+\frac{1}{2} \int_{B_{2R}}(|U|^2+|v|^2+|\Theta|^2)u \cdot \nabla \eta d x- \int_{B_{2R}}(u \cdot\nabla )w \cdot U dx+\int_{B_{2R}}(U \cdot v )v \cdot \nabla\eta dx\notag \\
		&-\int_{B_{2R}}(v \cdot\nabla)w\cdot v dx+\int_{B_{2R}}\Theta v\cdot\nabla \eta dx.\notag
	\end{align}
	With the help of the H\"{o}lder inequaity and the Poincar\'{e} inequality(see Lemma \ref{Lem2.2a}), we have
	\begin{align}\label{ine4.7}
		&\left|-\int_{B_{2R}}\left(\nabla \eta \cdot \nabla\right)U\cdot Ud x
		-\int_{B_{2R}}\nabla \Theta \cdot \left(\Theta\nabla \eta\right)d x
		\right|\notag\\
		\leq& CR^{-1}\|\nabla U\|_{L^2 (A_R)}\|U\|_{L^2 (A_R)}+CR^{-1}\|\nabla \Theta\|_{L^2 (A_R)}\|\Theta\|_{L^2 (A_R)}\\
        \leq& C\|\nabla U\|_{L^2 (A_R)}^2+C\|\nabla \Theta\|_{L^2 (A_R)}^2\notag\\
		=& C\|\nabla u\|_{L^2 (A_R)}^2+C\|\nabla \theta\|_{L^2 (A_R)}^2.\notag
	\end{align}
	Using the H\"{o}lder inequality, the Minkowski inequality and the Poincar\'{e} inequality, we derive
	\begin{align}\label{ine4.8}
		\left|\int_{B_{2R}}\left(\nabla \eta \cdot \nabla\right)v\cdot vd x
		\right|&\leq CR^{-1}\|\nabla v\|_{L^2(A_R)}\|v\|_{L^2(A_R)}\notag\\
		&\leq CR^{-1}\|\nabla v\|_{L^2(A_R)}\left(\|V\|_{L^2(A_R)}+ \|\overline{v}_R\|_{L^2(A_R)}\right)\\
		&\leq C\|\nabla v\|_{L^2 (A_R)}^2+CR^{-1}\|\nabla v\|_{L^2(A_R)}\cdot CR^{\frac{3}{2}-\frac{3}{q}}\|v\|_{L^q(A_R)}\notag\\
		&\leq C\|\nabla v\|_{L^2 (A_R)}^2+CR^{\frac{1}{2}-\frac{3}{q}}\|v\|_{L^q(A_R)}\|\nabla v\|_{L^2(A_R)}.\notag
	\end{align}
	Using the H\"{o}lder inequality, \eqref{ine4.4} and the Poincar\'{e} inequality, we get
	\begin{align}\label{ine4.9}
		\left|\int_{B_{2R}}\nabla U:\nabla w dx\right|&\leq\|\nabla U\|_{L^2(A_R)}\|\nabla w\|_{L^2(A_R)}\notag\\
		&\leq \|\nabla U\|_{L^2(A_R)}\cdot CR^{-1}\|U\|_{L^2(A_R)}\\
		&\leq C\|\nabla u\|_{L^2(A_R)}^2.\notag
	\end{align}
	Using the H\"{o}lder inequality, we have
	\begin{align}\label{ine4.13}
		&\left|\int_{B_{2R}}|U|^2u \cdot \nabla \eta d x\right|+\left|\int_{B_{2R}}|\Theta|^2 u \cdot \nabla \eta d x\right|\notag\\
		\leq& CR^{-1}\|u\|_{L^p(A_R)}\|U\|_{L^{2p'}(A_R)}^2+CR^{-1}\|u\|_{L^p(A_R)}\|\Theta\|_{L^{2p'}(A_R)}^2.
	\end{align}

	Using the Minkowski inequality and the H\"{o}lder inequality, we derive
	\begin{align}\label{ine4.10}
		\|U\|_{L^p(A_R)}&\leq \|u\|_{L^p(A_R)}+\|\overline{u}_R\|_{L^p(A_R)}\notag\\
		&\leq\|u\|_{L^p(A_R)}+CR^\frac{3}{p}|\overline{u}_R|\\
		&\leq C\|u\|_{L^p(A_R)}.\notag
	\end{align}
	Similarly, we can derive
	\begin{align}
		&\|V\|_{L^q(A_R)}\leq C\|v\|_{L^q(A_R)},\label{ine4.11}\\
        &\|\Theta\|_{L^r(A_R)}\leq C\|\theta\|_{L^r(A_R)}.\label{ine4.12}
	\end{align}

	Using the H\"{o}lder inequality,  \eqref{ine4.4}, \eqref{ine4.10} and the Minkowski inequality, we obtain
	\begin{align}\label{ine4.14}
		&\left|\int_{B_{2R}}|v|^2 u \cdot \nabla \eta d x\right|+\left|\int_{B_{2R}}(U \cdot v )v \cdot \nabla\eta dx\right|+\left|\int_{B_{2R}}(v \cdot\nabla)w\cdot v dx\right|\notag\\
		\leq& \left(CR^{-1}\|u\|_{L^p(A_R)}+CR^{-1}\|U\|_{L^p(A_R)}+\|\nabla w\|_{L^{p}(A_R)}\right)\|v\|_{L^{2p'}(A_R)}^2\notag\\
        \leq& CR^{-1}\|u\|_{L^p(A_R)}\|v\|_{L^{2p'}(A_R)}^2\notag\\
		\leq& CR^{-1}\|u\|_{L^p(A_R)}\left(\|V\|_{L^{2p'}(A_R)}^2+\|\overline{v}_R\|_{L^{2p'}(A_R)}^2 \right)\\
		\leq& CR^{-1}\|u\|_{L^p(A_R)}\|V\|_{L^{2p'}(A_R)}^2+CR^{-1}\|u\|_{L^p(A_R)}\cdot CR^{3-\frac{3}{p}-\frac{6}{q}}\|v\|_{L^q(A_R)}^2\notag\\
		=& CR^{-1}\|u\|_{L^p(A_R)}\|V\|_{L^{2p'}(A_R)}^2+CR^{2-\frac{3}{p}-\frac{6}{q}}\|u\|_{L^p(A_R)}\|v\|_{L^q(A_R)}^2.\notag
	\end{align}
	Using the H\"{o}lder inequality and  \eqref{ine4.4}, we obtain
	\begin{align}\label{ine4.15}
		\left|\int_{B_{2R}}(u \cdot\nabla )w \cdot U dx\right|&\leq \|u\|_{L^p(A_R)}\|\nabla w\|_{L^{2p'}(A_R)}\|U\|_{L^{2p'}(A_R)}\notag\\
		&\leq CR^{-1}\|u\|_{L^p(A_R)}\|U\|_{L^{2p'}(A_R)}^2.
	\end{align}
	Using the Minkowski inequality, the H\"{o}lder inequality and \eqref{ine4.12}, we derive
	\begin{align}\label{ine4.16}
		\left|\int_{B_{2R}}\Theta v\cdot\nabla \eta dx\right|&\leq \left|\int_{B_{2R}}\Theta V\cdot\nabla \eta dx\right|+\left|\int_{B_{2R}}\Theta \overline{v}_R\cdot\nabla \eta dx\right|\notag\\
		&\leq CR^{-1}\|V\|_{L^2(A_R)}\|\Theta\|_{L^2(A_R)}+CR^{-1}\|\Theta\|_{L^r(A_R)}\|\overline{v}_R\|_{L^{r'}(A_R)}\\
		&\leq CR^{-1}\|V\|_{L^2(A_R)}\|\Theta\|_{L^2(A_R)}+CR^{2-\frac{3}{q}-\frac{3}{r}}\|v\|_{L^{q}(A_R)}\|\theta\|_{L^r(A_R)}.\notag
	\end{align}
	Combining \eqref{ine4.6}, \eqref{ine4.7}, \eqref{ine4.8}, \eqref{ine4.9}, \eqref{ine4.13}, \eqref{ine4.14}, \eqref{ine4.15} and \eqref{ine4.16}, we conclude that \eqref{ine4.3} holds.
\end{proof}

\begin{Rem}\label{Rem4.1}
In the proof of {\rm Lemma \ref{Lem4.2}}, we transform \eqref{equ1.1} by variable substitutions
$$U=u-\overline{u}_R,\;\Theta=\theta-\overline{\theta}_R,$$
but we do not make the translation transformation for $v$. Actually, if we also perform the translation transformation for $v$:
$$V=v-\overline{v}_R,$$
the extra term
$$
\int_{B_{2R}}\overline{v}_R\otimes V:(\eta\nabla U)dx
$$
originating from $\divg(v\otimes v)$ can neither be cancelled out by other terms nor be properly controlled.
It is this asymmetric treatment that leads to the breakdown of the symmetric role between $v$ and $\theta$. Compared with $\theta$, we need to deal with additional terms $K_2$ and $K_6$ for $v$, where the definitions of $K_2$ and $K_6$ are given below.
\end{Rem}

For the sake of convenience, we denote the eight terms on the right hand side of \eqref{ine4.3} by $K_1$, $K_2$, $\cdots$, $K_8$, respectively, i.e.
	\begin{align*}
		&K_1=C\left(\|\nabla u\|_{L^2(A_R)}^2+\|\nabla v\|_{L^2(A_R)}^2+\|\nabla \theta\|_{L^2(A_R)}^2\right),\;K_2=CR^{\frac{1}{2}-\frac{3}{q}}\|v\|_{L^q(A_R)}\|\nabla v\|_{L^2(A_R)},\\
		&K_3=CR^{-1}\|u\|_{L^p(A_R)}\|U\|_{L^{2p'}(A_R)}^2,\;K_4=CR^{-1}\|u\|_{L^p(A_R)}\|V\|_{L^{2p'}(A_R)}^2,\\
		&K_5=CR^{-1}\|u\|_{L^p(A_R)}\|\Theta\|_{L^{2p'}(A_R)}^2,\;K_6=CR^{2-\frac{3}{p}-\frac{6}{q}}\|u\|_{L^p(A_R)}\|v\|_{L^q(A_R)}^{2},\\
		&K_7=CR^{-1}\|V\|_{L^2(A_R)}\|\Theta\|_{L^2(A_R)},\;K_8=CR^{2-\frac{3}{q}-\frac{3}{r}}\|v\|_{L^{q}(A_R)}\|\theta\|_{L^r(A_R)}.
	\end{align*}
In order to control $K_1$, we need to estimate $f(R)$, which is defined by \eqref{ine3.1}.

\begin{Lem}\label{Lem4.3}
Let the assumptions be the same as those in {\rm Theorem \ref{main2}}. Then there exist three positive constants $R_1>3$, $A$ and $C$ such that
	\begin{align}\label{ine4.17}
		f(R)\leq C(\ln R)^{A},\;\forall R\geq R_1.
	\end{align}
\begin{proof}
No matter whether $\mathrm{(B1)}$ or $\mathrm{(B2)}$ holds, there always exist two positive constants $R_1>3$ and $C$ such that the following three inequalities hold for any $R\geq R_1$:	
\begin{align}
&\|u\|_{L^p\left(A_R\right)}\leq CR^\alpha(\ln R)^\lambda,\label{add4.19}\\
&\|v\|_{L^q\left(A_R\right)}\leq CR^\beta(\ln R)^\mu,\label{add4.20}\\
&\|\theta\|_{L^r\left(A_R\right)}\leq CR^\gamma(\ln R)^\nu.\label{add4.21}
\end{align}
We only consider the case $p,q,r\in(\frac{3}{2},3)\times[1,2)\times[1,2)$ as the remaining cases can be treated similarly.	
Combining \eqref{add1}, \eqref{add4.19}, \eqref{add4.20} and \eqref{add4.21}, we obtain
	\begin{align*}
		f(R)\leq&CR^{1-\frac{6}{p}+2\alpha}(\ln R)^{2\lambda}+CR^{1-\frac{6}{q}+2\beta}(\ln R)^{2\mu}+CR^{1-\frac{6}{r}+2\gamma}(\ln R)^{2\nu}\\
		+&CR^{\frac{p-6}{2p-3}+\frac{3p}{2p-3}\alpha}(\ln R)^{\frac{3p}{2p-3}\lambda}+CR^{-\frac{(6-q)p'}{(3-p')q}+\frac{(6-q)p'}{(3-p')q}\alpha+2\beta}(\ln R)^{\frac{(6-q)p'}{(3-p')q}\lambda+2\mu}\\
+&CR^{-\frac{(6-r)p'}{(3-p')r}+\frac{(6-r)p'}{(3-p')r}\alpha+2\gamma}(\ln R)^{\frac{(6-r)p'}{(3-p')r}\lambda+2\nu}\\
		+&CR^\frac{-(6-q)(6-r)+2q(6-r)\beta+2r(6-q)\gamma}{q(6-r)+r(6-q)}(\ln R)^{{\frac{2q(6-r)\mu+2r(6-q)\nu}{q(6-r)+r(6-q)}}}.
	\end{align*}
Observing that the exponents of $R$ in the above inequality are all less than or equal to $0$,  we obtain further
\begin{align*}
	f(R)\leq&C(\ln R)^{2\lambda}+C(\ln R)^{2\mu}+C(\ln R)^{2\nu}+C(\ln R)^{\frac{3p}{2p-3}\lambda}
	+C(\ln R)^{\frac{(6-q)p'}{(3-p')q}\lambda+2\mu}\\
	&+C(\ln R)^{\frac{(6-r)p'}{(3-p')r}\lambda+2\nu}
	+C(\ln R)^{{\frac{2q(6-r)\mu+2r(6-q)\nu}{q(6-r)+r(6-q)}}}\\
	\leq&C(\ln R)^{\frac{3p}{2p-3}\lambda}+C(\ln R)^{\frac{(6-q)p'}{(3-p')q}\lambda+2\mu}+C(\ln R)^{\frac{(6-r)p'}{(3-p')r}\lambda+2\nu}+C(\ln R)^{{\frac{2q(6-r)\mu+2r(6-q)\nu}{q(6-r)+r(6-q)}}}.
\end{align*}
We choose
	\begin{align*}
		A=\max\left\{{\frac{3p\lambda}{2p-3},\;\frac{(6-q)p'}{(3-p')q}\lambda+2\mu,\;\frac{(6-r)p'\lambda}{(3-p')r}+2\nu},\;{\frac{2q(6-r)\mu+2r(6-q)\nu}{q(6-r)+r(6-q)}},\;1\right\},
	\end{align*}
	and then we get a concise estimate for $f(R)$:
	\begin{align*}
		f(R)\leq C(\ln R)^{A}.
	\end{align*}
\end{proof}
\end{Lem}

With the help of Lemma \ref{Lem4.3}, we are able to control $K_1$.

\begin{Lem}\label{Lem4.4}
Let the assumptions be the same as those in  {\rm Theorem \ref{main2}}. Let $R_1$ and $A$ be the constants given in {\rm Lemma \ref{Lem4.3}}. Let $\tau$ be the maximum among the following six numbers
\begin{align*}	
&\frac{A}{A+1},\;\frac{1}{2},\;\frac{9-3p}{6-p},\;\frac{6p-3\min\{q,2p'\}(p-1)}{(6-\min\{q,2p'\})p},\;\frac{6p-3\min\{r,2p'\}(p-1)}{(6-\min\{r,2p'\})p},\\
&\frac{6-3\min\{q,2\}}{12-2\min\{q,2\}}+\frac{6-3\min\{r,2\}}{12-2\min\{r,2\}}.
\end{align*}
Then there exists a positive constant $C$ such that
	\begin{align}\label{ine4.22}
	K_1\leq&C\left[R\ln R E'(R)\right]^\tau,\;\forall R\geq R_1.
	\end{align}
\end{Lem}
\begin{proof}
	It is not difficult to verify that $\tau\in(0,1)$ and $A(1-\tau)\leq\tau$.
	Using \eqref{ine4.2} and \eqref{ine4.17}, we have
	\begin{align*}
		K_1=K_1^{1-\tau}\cdot K_1^\tau \leq&\left[Cf(2R)\right]^{1-\tau}\cdot\left[CR E'(R)\right]^{\tau}\notag\\
		\leq& C(\ln R)^{A(1-\tau)}\left[R E'(R)\right]^\tau\\
		\leq&C\left[R\ln R E'(R)\right]^\tau,\notag
	\end{align*}
	where we require $R\geq R_1$.
	\end{proof}

The estimates of the remaining seven terms $K_2$, $K_3$, $\cdots$, $K_8$ will be given in the subsequent seven lemmas.

\begin{Lem}\label{Lem4.5}
Let the assumptions be the same as those in  {\rm Theorem \ref{main2}}. Then there exists a positive constant $C$ such that
\begin{equation}\label{ine4.23}
K_2\leq C[R\ln R E'(R)]^\frac{1}{2},\;\forall R\geq R_1.
\end{equation}
\end{Lem}
\begin{proof}
When $\beta< \frac{3}{q}-\frac{1}{2}$, using \eqref{ine4.2}, for any $R\geq R_1$, we have
\begin{align*}
K_2&\leq CR^{\beta-\left(\frac{3}{q}-\frac{1}{2}\right)}(\ln R)^\mu\|\nabla v\|_{L^2(A_R)}\\
&\leq C\|\nabla v\|_{L^2(A_R)}\\
&\leq C[R E'(R)]^\frac{1}{2}\\
&\leq C[R\ln R E'(R)]^\frac{1}{2}.
\end{align*}
When $\beta=\frac{3}{q}-\frac{1}{2}$, using \eqref{ine4.2}, for any $R\geq R_1$, we have
\begin{align*}
K_2&\leq CR^{\beta-\left(\frac{3}{q}-\frac{1}{2}\right)}(\ln R)^\mu\|\nabla v\|_{L^2(A_R)}\\
   &\leq C(\ln R)^\mu[R E'(R)]^\frac{1}{2}\\
   &\leq C[R\ln R E'(R)]^\frac{1}{2}.
\end{align*}
Here we have used the fact that when $\beta=\frac{3}{q}-\frac{1}{2}$, we require $0\leq \mu\leq \frac{1}{2}$.
\end{proof}

\begin{Lem}\label{Lem4.6}
Let the assumptions be the same as those in {\rm Theorem \ref{main2}}. Assume $E(R)\not\equiv0$. Then there exist two positive constants $R_4>3$ and $C$ such that
\begin{align}\label{ine4.24}
K_3\leq \frac{1}{16}E(R)+C\chi_1(p)\left[R\ln R E'(R)\right]^{\frac{9-3p}{6-p}\chi_1(p)},\;\forall R\geq R_4,
\end{align}
where $\chi_1(p)$ is defined by
\begin{equation*}
\chi_1(p)=
\begin{cases}
1, & \text{ if  }\frac{3}{2}<p<3, \\
0, & \text{ if  }3\leq p\leq\frac{9}{2}.
\end{cases}
\end{equation*}
\end{Lem}
\begin{proof}
	When $p\in\left(\frac{3}{2},3\right)$, using the interpolation inequality, \eqref{ine4.10}, the Sobolev-Poincar\'{e} inequality(see Lemma \ref{Lem2.2}), \eqref{add4.19} and \eqref{ine4.2}, we obtain
	\begin{align}\label{ine4.25}
		K_3&\leq CR^{-1}\|u\|_{L^p(A_R)}\left(\|U\|_{L^p(A_R)}^{\frac{2p-3}{6-p}}\|U\|_{L^6(A_R)}^{\frac{9-3p}{6-p}}\right)^2\notag\\
		&\leq CR^{-1}\|u\|_{L^p(A_R)}^{\frac{3p}{6-p}}\|\nabla u\|_{L^2(A_R)}^{\frac{2(9-3p)}{6-p}}\\
        &\leq CR^{\frac{3p}{6-p}\alpha-1}(\ln R)^{\frac{3p}{6-p}\lambda}[RE'(R)]^\frac{9-3p}{6-p}\notag\\
        &\leq C[R\ln RE'(R)]^\frac{9-3p}{6-p},\notag
	\end{align}
   for any $R\geq R_1$.

	When $p\in[3,\frac{9}{2}]$, by applying the H\"{o}lder inequality and using \eqref{ine4.10}, we obtain
	\begin{align*}
	K_3&\leq CR^{-1}\|u\|_{L^p(A_R)}\left(\|U\|_{L^p(A_R)}CR^{3\left(\frac{1}{2p'}-\frac{1}{p}\right)}\right)^2\\
       &= CR^{2-\frac{9}{p}}\|u\|_{L^p(A_R)}\|U\|_{L^p(A_R)}^2\\
       &\leq CR^{2-\frac{9}{p}}\|u\|_{L^p(A_R)}^3\\
       &=CR^{2-\frac{9}{p}+3\alpha}X_{p,\alpha}^3(R)\\
       &\leq CX_{p,\alpha}^3(R),
	\end{align*}
for any $R\geq R_1$. Hence, we have $\lim\limits_{R\rightarrow+\infty}K_3=0$ at this moment. Since $E(R)\not\equiv0$, in view of the non-decreasing property of $E(R)$, there exists a constant $R_2>R_1$ such that
\begin{align}\label{ine4.26}
\text{$E(R)\geq E(R_2)>0$ for any $R\geq R_2$.}
\end{align}
Considering $\lim\limits_{R\rightarrow+\infty}K_3=0$, there exists a constant $R_3>R_2$ such that
\begin{align}\label{ine4.27}
K_3\leq\frac{1}{16}E(R_2)\leq \frac{1}{16}E(R),\;\forall R\geq R_3.
\end{align}

Denote $R_4=\chi_1(p)R_1+[1-\chi_1(p)]R_3$. Taking \eqref{ine4.25} and \eqref{ine4.27} into consideration together, we find that \eqref{ine4.24} holds.
\end{proof}

\begin{Lem}\label{Lem4.7}
Let the assumptions be the same as those in  {\rm Theorem \ref{main2}}. Assume $E(R)\not\equiv0$.
Then there exist two positive constants $R_6>3$ and $C$ such that
\begin{align}\label{ine4.28}
K_4\leq \frac{1}{16}E(R)+C\chi_2(q)\left[R\ln R E'(R)\right]^\frac{6p-3\min\{q,2p'\}(p-1)}{(6-\min\{q,2p'\})p},\;\forall R\geq R_6,
\end{align}
where $\chi_2(q)$ is defined by
\begin{equation*}
\chi_2(q)=
\begin{cases}
1, & \text{ if  }1\leq q<2p', \\
0, & \text{ if  }2p'\leq q\leq6.
\end{cases}
\end{equation*}
\end{Lem}
\begin{proof}
	When $q\in[1,2p')$, using the interpolation inequality, \eqref{ine4.11} and the Sobolev-Poincar\'{e} inequality, \eqref{add4.19}, \eqref{add4.20} and \eqref{ine4.2}, we obtain
	\begin{align}\label{ine4.29}
		K_4&\leq CR^{-1}\|u\|_{L^{p}(A_R)}\|V\|_{L^q(A_R)}^\frac{(6-2p')q}{(6-q)p'}\|V\|_{L^6(A_R)}^\frac{12p'-6q}{(6-q)p'}\notag\\
		&\leq CR^{-1}\|u\|_{L^{p}(A_R)}\|v\|_{L^q(A_R)}^\frac{(4p-6)q}{(6-q)p}\|\nabla v\|_{L^2(A_R)}^\frac{12p-6q(p-1)}{(6-q)p}\\
        &\leq CR^{\alpha+\frac{(4p-6)q}{(6-q)p}\beta-1}(\ln R)^{\lambda+\frac{(4p-6)q}{(6-q)p}\mu}[RE'(R)]^\frac{6p-3q(p-1)}{(6-q)p}\notag\\
        &\leq C\left[R\ln R E'(R)\right]^\frac{6p-3q(p-1)}{(6-q)p}.\notag
	\end{align}
	
	When $q\in [2p',6]$, using the H\"{o}lder inequality, \eqref{ine4.11}, \eqref{add4.19} and \eqref{add4.20}, we obtain
	\begin{align*}
		K_4&\leq CR^{2-\frac{3}{p}-\frac{6}{q}}\|u\|_{L^p(A_R)}\|V\|_{L^q(A_R)}^2\notag\\
		&\leq CR^{2-\frac{3}{p}-\frac{6}{q}}\|u\|_{L^p(A_R)}\|v\|_{L^q(A_R)}^2\\
        &\leq CR^{2-\frac{3}{p}-\frac{6}{q}+\alpha+2\beta}(\ln R)^{\lambda+2\mu}.\notag
	\end{align*}
Since $2-\frac{3}{p}-\frac{6}{q}+\alpha+2\beta<0$ at this moment, we have $\lim\limits_{R\rightarrow+\infty}K_4=0$. Hence, there exists a constant $R_5>R_2$ such that
\begin{align}\label{ine4.30}
K_4\leq\frac{1}{16}E(R_2)\leq \frac{1}{16}E(R),\;\forall R\geq R_5.
\end{align}

Denote $R_6=\chi_2(q)R_1+[1-\chi_2(q)]R_5$. Taking \eqref{ine4.29} and \eqref{ine4.30} into consideration together, we find that \eqref{ine4.28} holds.
\end{proof}

Following a similar argument as in the proof of Lemma \ref{Lem4.7}, we obtain Lemma \ref{Lem4.8}.

\begin{Lem}\label{Lem4.8}
Let the assumptions be the same as those in  {\rm Theorem \ref{main2}}. Assume $E(R)\not\equiv0$. Then there exist two positive constants $R_7>3$ and $C$ such that
\begin{align*}
K_5\leq \frac{1}{16}E(R)+C\chi_2(r)\left[R\ln R E'(R)\right]^\frac{6p-3\min\{r,2p'\}(p-1)}{(6-\min\{r,2p'\})p},\;\forall R\geq R_7.
\end{align*}
\end{Lem}

\begin{Lem}\label{Lem4.9}
Let the assumptions be the same as those in {\rm Theorem \ref{main2}}. Assume $E(R)\not\equiv0$. Then there exist two positive constants $R_8$ and $C$ such that
\begin{equation*}
K_{6}\leq\frac{1}{16}E(R),\;\forall R\geq R_8.
\end{equation*}
\end{Lem}
\begin{proof}
For $(p,q)\in\left(3,\frac{9}{2}\right]\times[1,2p')$, we have $\alpha\geq0>\frac{3}{p}-1$ and $(\alpha,\beta)\neq\left(\frac{3}{p}-1,\frac{3}{q}-\frac{1}{2}\right)$.
Thus, together with our assumption, we find that $(\alpha,\beta)\neq\left(\frac{3}{p}-1,\frac{3}{q}-\frac{1}{2}\right)$ for any $(p,q)\in\left(\frac{3}{2},\frac{9}{2}\right]\times[1,2p')$.
When $(p,q)\in\left(\frac{3}{2},\frac{9}{2}\right]\times[1,2p')$, it holds that
\begin{align*}
2-\frac{3}{p}-\frac{6}{q}+\alpha+2\beta&\leq 2-\frac{3}{p}-\frac{6}{q}+\left[1-\frac{(4p-6)q}{(6-q)p}\beta\right]+2\beta\\
&= 3-\frac{3}{p}-\frac{6}{q}+\frac{6(p-1)(2p'-q)}{(6-q)p}\beta\\
&\leq3-\frac{3}{p}-\frac{6}{q}+\frac{6(p-1)(2p'-q)}{(6-q)p}\left(\frac{3}{q}-\frac{1}{2}\right)\\
&=0.
\end{align*}
We claim that $2-\frac{3}{p}-\frac{6}{q}+\alpha+2\beta$ can not attain zero, otherwise, we have
\begin{align*}
\begin{cases}
\beta=\frac{3}{q}-\frac{1}{2}\\
\alpha=1-\frac{(4p-6)q}{(6-q)p}\beta=\frac{3}{p}-1,
\end{cases}
\end{align*}
which is impossible.

Combining the above conclusion for $(p,q)\in\left(\frac{3}{2},\frac{9}{2}\right]\times[1,2p')$ and the assumption $\alpha+2\beta<\frac{3}{p}+\frac{6}{q}-2$ for $(p,q)\in\left(\frac{3}{2},\frac{9}{2}\right]\times[2p',6]$, we conclude that
$$
2-\frac{3}{p}-\frac{6}{q}+\alpha+2\beta<0
$$
always holds. Consequently, we have
$$
\lim\limits_{R\rightarrow+\infty}K_6=\lim\limits_{R\rightarrow+\infty}CR^{2-\frac{3}{p}-\frac{6}{q}+\alpha+2\beta}(\ln R)^{\lambda+2\mu}X_{p,\alpha,\lambda}(R)Y^2_{q,\beta,\mu}(R)=0.
$$
Hence, there exists a constant $R_8>R_2$ such that
\begin{align*}
K_6\leq\frac{1}{16}E(R_2)\leq \frac{1}{16}E(R),\;\forall R\geq R_8.
\end{align*}
\end{proof}

\begin{Lem}\label{Lem4.10}
	Let the assumptions be the same as those in {\rm Theorem \ref{main2}}. Assume $E(R)\not\equiv0$.
	There exist two positive constants $R_{10}>3$ and $C$ such that
\begin{align}\label{ine4.31}
K_7\leq \frac{1}{16}E(R)+C\chi_3(q,r)\left[R\ln R E'(R)\right]^{\frac{6-3\min\{q,2\}}{12-2\min\{q,2\}}+\frac{6-3\min\{r,2\}}{12-2\min\{r,2\}}},\;\forall R\geq R_{10},
\end{align}
where $\chi_3(q,r)$ is defined by
\begin{equation*}
\chi_3(q,r)=
\begin{cases}
1, & \text{ if  }(q,r)\in [1,6]\times[1,6] \text{ and }(q,r)\notin[2,6]\times[2,6], \\
0, & \text{ if  }(q,r)\in [2,6]\times[2,6].
\end{cases}
\end{equation*}
\end{Lem}
\begin{proof}
Similarly to \eqref{ine3.24}, \eqref{ine3.26}, \eqref{ine3.28} and \eqref{ine3.29} in Lemma \ref{Lem2.9}, we can obtain
\begin{itemize}
		\item[(i)] When $1\leq q<2, 1\leq r<2$, it holds that
		\begin{equation*}
			K_{7}
			\leq CR^{-1}\|V\|_{L^q(A_R)}^{\frac{2q}{6-q}}\|\Theta\|_{{L^r}(A_R)}^{\frac{2r}{6-r}}\|V\|_{L^6(A_R)}^{\frac{3(2-q)}{6-q}}\|\Theta\|_{L^6(A_R)}^\frac{3(2-r)}{6-r};
		\end{equation*}
		\item[(ii)]  When $1\leq q<2$, $2\leq r\leq6$, it holds that
		\begin{equation*}
			K_{7}
			\leq CR^{\frac{1}{2}-\frac{3}{r}}\|V\|_{L^q(A_R)}^{\frac{2q}{6-q}} \|\Theta\|_{L^{r}(A_R)}\|V\|_{L^6(A_R)}^{\frac{3(2-q)}{6-q}};
		\end{equation*}
	\item[(iii)] When $2\leq q\leq6$, $1\leq r<2$, it holds that
	\begin{equation*}
		K_{7}
		\leq CR^{\frac{1}{2}-\frac{3}{q}}\|V\|_{L^{q}(A_R)}\|\Theta\|_{{L^r}(A_R)}^{\frac{2r}{6-r}}\|\Theta\|_{{L^6}(A_R)}^\frac{3(2-r)}{6-r};
	\end{equation*}
	\item[(iv)]  When $2\leq q\leq6$, $2\leq r\leq 6$, it holds that
\begin{equation*}
	K_{7}
	\leq CR^{2-\frac{3}{q}-\frac{3}{r}}\|V\|_{L^{q}(A_R)}\|\Theta\|_{L^{r}(A_R)}.
\end{equation*}
	\end{itemize}

Combining \eqref{ine4.11}, \eqref{ine4.12}, the Sobolev-Poincar\'{e} inequality, \eqref{add4.20}, \eqref{add4.21} and \eqref{ine4.2}, and using Assumptions we conclude
\begin{itemize}
		\item[(i)] When $1\leq q<2, 1\leq r<2$, it holds that
		\begin{align}\label{ine4.32}
		K_{7}&\leq CR^{-1}\|v\|_{L^q(A_R)}^{\frac{2q}{6-q}}\|\theta\|_{{L^r}(A_R)}^{\frac{2r}{6-r}}\|\nabla v\|_{L^2(A_R)}^{\frac{3(2-q)}{6-q}}\|\nabla\theta\|_{L^2(A_R)}^\frac{3(2-r)}{6-r}\notag\\
             &\leq CR^{\frac{2q}{6-q}\beta+\frac{2r}{6-r}\gamma-1}(\ln R)^{\frac{2q}{6-q}\mu+\frac{2r}{6-r}\nu}[RE'(R)]^{\frac{3(2-q)}{2(6-q)}+\frac{3(2-r)}{2(6-r)}}\\
             &\leq C[R\ln RE'(R)]^{\frac{3(2-q)}{2(6-q)}+\frac{3(2-r)}{2(6-r)}};\notag
		\end{align}
		\item[(ii)]  When $1\leq q<2$, $2\leq r\leq6$, it holds that
		\begin{align}\label{ine4.33}
			K_{7}&\leq CR^{\frac{1}{2}-\frac{3}{r}}\|v\|_{L^q(A_R)}^{\frac{2q}{6-q}} \|\theta\|_{L^{r}(A_R)}\|\nabla v\|_{L^2(A_R)}^{\frac{3(2-q)}{6-q}}\notag\\
                 &\leq CR^{\frac{2q}{6-q}\beta+\gamma+\frac{1}{2}-\frac{3}{r}}(\ln R)^{\frac{2q}{6-q}\mu+\nu}[RE'(R)]^\frac{3(2-q)}{2(6-q)}\\
                 &\leq C[R\ln RE'(R)]^\frac{3(2-q)}{2(6-q)};\notag
		\end{align}
	\item[(iii)] When $2\leq q\leq6$, $1\leq r<2$, it holds that
	\begin{align}\label{ine4.34}
		K_{7}&\leq CR^{\frac{1}{2}-\frac{3}{q}}\|v\|_{L^{q}(A_R)}\|\theta\|_{{L^r}(A_R)}^{\frac{2r}{6-r}}\|\nabla\theta\|_{{L^2}(A_R)}^\frac{3(2-r)}{6-r}\notag\\
             &\leq CR^{\beta+\frac{2r}{6-r}\gamma+\frac{1}{2}-\frac{3}{q}}(\ln R)^{\mu+\frac{2r}{6-r}\nu}[RE'(R)]^\frac{3(2-r)}{2(6-r)}\\
             &\leq C[R\ln RE'(R)]^\frac{3(2-r)}{2(6-r)};\notag
	\end{align}
	\item[(iv)]  When $2\leq q\leq6$, $2\leq r\leq 6$, it holds that
\begin{align*}
	K_{7}\leq CR^{2-\frac{3}{q}-\frac{3}{r}}\|v\|_{L^{q}(A_R)}\|\theta\|_{L^{r}(A_R)}\leq CR^{2-\frac{3}{q}-\frac{3}{r}+\beta+\gamma}(\ln R)^{\mu+\nu}.
\end{align*}
	\end{itemize}
Since $2-\frac{3}{q}-\frac{3}{r}+\beta+\gamma<0$ if $2\leq q\leq6$ and $2\leq r\leq 6$, we have $\lim\limits_{R\rightarrow+\infty}K_7=0$ at this moment. Hence, there exists a constant $R_9>R_2$ such that
\begin{align}\label{ine4.35}
K_7\leq\frac{1}{16}E(R_2)\leq \frac{1}{16}E(R),\;\forall R\geq R_9.
\end{align}

Denote $R_{10}=\chi_3(q,r)R_1+[1-\chi_3(q,r)]R_9$. Taking \eqref{ine4.32}, \eqref{ine4.33}, \eqref{ine4.34} and \eqref{ine4.35} into consideration together, we find that \eqref{ine4.31} holds.
\end{proof}

\begin{Lem}\label{Lem4.11}
Let the assumptions be the same as those in {\rm Theorem \ref{main2}}. Assume $E(R)\not\equiv0$.
Then there exist two positive constants $R_{11}$ and $C$ such that
\begin{equation*}
K_8\leq\frac{1}{16}E(R),\;\forall R\geq R_{11}.
\end{equation*}
\end{Lem}
\begin{proof}
Suppose that $q\in\left[1,2\right)$ and $r\in[1,2)$. At this moment, we have
\begin{align*}
2-\frac{3}{q}-\frac{3}{r}+\beta+\gamma&\leq2-\frac{3}{q}-\frac{3}{r}+\frac{6-q}{2q}\left(1-\frac{2r}{6-r}\gamma\right)+\gamma\\
&=\frac{3}{2}-\frac{3}{r}+\frac{6(q-r)}{(6-r)q}\gamma.
\end{align*}
When $q\leq r$, it is obvious that $2-\frac{3}{q}-\frac{3}{r}+\beta+\gamma<0$. When $q>r$, we find
\begin{align*}
2-\frac{3}{q}-\frac{3}{r}+\beta+\gamma&\leq\frac{3}{2}-\frac{3}{r}+\frac{6(q-r)}{(6-r)q}\cdot\left(\frac{3}{r}-\frac{1}{2}\right)=\frac{3}{2}-\frac{3}{q}<0.
\end{align*}

Suppose that $q\in\left[1,2\right)$ and $r\in[2,6]$. At this moment, we have
\begin{align*}
2-\frac{3}{q}-\frac{3}{r}+\beta+\gamma&\leq2-\frac{3}{q}-\frac{3}{r}+\beta+\left(\frac{6-r}{2r}-\frac{2q}{6-q}\beta\right)\\
&=\frac{3}{2}-\frac{3}{q}+\frac{6-3q}{6-q}\beta\\
&\leq\frac{3}{2}-\frac{3}{q}+\frac{6-3q}{6-q}\cdot\left(\frac{3}{q}-\frac{1}{2}\right)\\
&=0.
\end{align*}
We claim that $2-\frac{3}{q}-\frac{3}{r}+\beta+\gamma$ can not attain zero, otherwise, we have
\begin{align*}
\begin{cases}
\beta=\frac{3}{q}-\frac{1}{2}\\
\gamma=\frac{6-r}{2r}-\frac{2q}{6-q}\beta=\frac{3}{r}-\frac{3}{2}\leq0.
\end{cases}
\end{align*}
Using the fact $\gamma\geq0$, we derive $(r,\beta,\gamma)=(2,\frac{3}{q}-\frac{1}{2},0)$, which is impossible.

Suppose that $q\in[2,6]$ and $r\in\left[1,2\right)$. By the same method as in the previous case, we obtain
\begin{align}\label{ine4.36}
2-\frac{3}{q}-\frac{3}{r}+\beta+\gamma<0.
\end{align}

Suppose that $q\in[2,6]$ and $r\in[2,6]$. At this moment, Assumption \ref{a1.2}(iii)  becomes
$$
\beta+\gamma<\frac{3}{q}+\frac{3}{r}-2.
$$

In conclusion, \eqref{ine4.36} always holds. As a direct consequence, we get
\begin{equation*}
	\lim\limits_{R\rightarrow+\infty}K_{8}=\lim\limits_{R\rightarrow+\infty}CR^{2-\frac{3}{q}-\frac{3}{r}+\beta+\gamma}(\ln R)^{\mu+\nu}Y_{q,\beta,\mu}(R)Z_{r,\gamma,\nu}(R)=0.
\end{equation*}
Thus, there exists a constant $R_{11}>R_2$ such that
\begin{align*}
K_8\leq\frac{1}{16}E(R_2)\leq \frac{1}{16}E(R),\;\forall R\geq R_{11}.
\end{align*}
\end{proof}

Now we are ready to prove Theorem \ref{main2}.

\begin{proof}[{\bf Proof of Theorem \ref{main2}}]
Firstly, we claim that $E(R)\equiv0$, and we prove this by contradiction.
Assume that $E(R)\not\equiv0$. Then \eqref{ine4.26} holds. Using Lemmas \ref{Lem4.2}, \ref{Lem4.4}, \ref{Lem4.5}, \ref{Lem4.6}, \ref{Lem4.7}, \ref{Lem4.8}, \ref{Lem4.9}, \ref{Lem4.10} and \ref{Lem4.11}, we derive
\begin{align}\label{ine3.37}
E(R)\leq&\frac{1}{2}E(R)+C\left[R\ln R E'(R)\right]^\tau+C[R\ln R E'(R)]^\frac{1}{2}\notag\\
&+C\chi_1(p)\left[R\ln R E'(R)\right]^{\frac{9-3p}{6-p}\chi_1(p)}+C\chi_2(q)\left[R\ln R E'(R)\right]^\frac{6p-3\min\{q,2p'\}(p-1)}{(6-\min\{q,2p'\})p}\\
&+C\chi_2(r)\left[R\ln R E'(R)\right]^\frac{6p-3\min\{r,2p'\}(p-1)}{(6-\min\{r,2p'\})p}\notag\\
&+C\chi_3(q,r)\left[R\ln R E'(R)\right]^{\frac{6-3\min\{q,2\}}{12-2\min\{q,2\}}+\frac{6-3\min\{r,2\}}{12-2\min\{r,2\}}},\notag
\end{align}
where $R\geq R_0:=\max\{R_4,R_6,R_7,R_8,R_{10},R_{11}\}$. By the Young inequality, we obtain
\begin{align}\label{ine3.38}
	C[R\ln R E'(R)]^\frac{1}{2}&\leq \frac{1}{20}E(R_2)+C[R\ln R E'(R)]^\tau,
\end{align}
\begin{align}\label{ine3.39}
	C\chi_1(p)[R\ln R E'(R)]^{\frac{9-3p}{6-p}\chi_1(p)}&\leq \frac{1}{20}E(R_2)+C[R\ln R E'(R)]^\tau,
\end{align}
\begin{align}\label{ine3.40}
	C\chi_2(q)[R\ln R E'(R)]^\frac{6p-3\min\{q,2p'\}(p-1)}{(6-\min\{q,2p'\})p}&\leq \frac{1}{20}E(R_2)+C[R\ln R E'(R)]^\tau,
\end{align}
\begin{align}\label{ine3.41}
	C\chi_2(r)[R\ln R E'(R)]^\frac{6p-3\min\{r,2p'\}(p-1)}{(6-\min\{r,2p'\})p}&\leq \frac{1}{20}E(R_2)+C[R\ln R E'(R)]^\tau,
\end{align}
\begin{align}\label{ine3.42}
	C\chi_3(q,r)\left[R\ln R E'(R)\right]^{\frac{6-3\min\{q,2\}}{12-2\min\{q,2\}}+\frac{6-3\min\{r,2\}}{12-2\min\{r,2\}}}&\leq \frac{1}{20}E(R_2)+C[R\ln R E'(R)]^\tau.
\end{align}
Combining \eqref{ine3.37}, \eqref{ine3.38}, \eqref{ine3.39}, \eqref{ine3.40}, \eqref{ine3.41} and \eqref{ine3.42}, we obtain
\begin{align*}
E(R)&\leq \frac{1}{2}E(R)+\frac{1}{4}E(R_2)+C[R\ln R E'(R)]^\tau\\
&\leq \frac{3}{4}E(R)+C[R\ln R E'(R)]^\tau,
\end{align*}
	which implies
	\begin{align*}
		E(R)
		&\leq C\left[R\ln RE'(R)\right]^\tau.
	\end{align*}
	Consequently, it follows that
	\begin{align*}
		\ln\ln R-\ln\ln R_0=\int_{R_0}^{R}\frac{1}{\rho\ln\rho}d\rho\leq\int_{R_0}^{R}\frac{CE'(\rho)}{E^{\frac{1}{\tau}}(\rho)}d\rho\leq CE(R_0)^{1-\frac{1}{\tau}}<+\infty.
	\end{align*}
	Letting $R\rightarrow+\infty$, the above inequality leads to a contradiction. Therefore, $E(R)\equiv0$.
	
	Thanks to the simple inequality
	$$\|\nabla u\|_{L^2(B_R)}^2+\|\nabla v\|_{L^2(B_R)}^2+\|\nabla \theta\|_{L^2(B_R)}^2\leq E(R),$$
	we conclude that $u,v$ are constant vectors and $\theta$ is a constant. Finally, the limit conditions in $\mathrm{(B1)}$ or $\mathrm{(B2)}$, and the following three equalities
\begin{align*}
&|u|=CR^{\alpha-\frac{3}{p}}(\ln R)^\lambda\cdot X_{p,\alpha,\lambda}(R),\\
&|v|=CR^{\beta-\frac{3}{q}}(\ln R)^\mu\cdot Y_{q,\beta,\mu}(R),\\
&|\theta|=CR^{\gamma-\frac{3}{r}}(\ln R)^\nu\cdot Z_{r,\gamma,\nu}(R)
\end{align*}
force $u,v,\theta$ to be zero.
\end{proof}

\section{Proofs of Theorem \ref{main3} and Theorem \ref{main4}}\label{sec5}
Let $\Theta$ be defined by \eqref{eq4.3}. Denote
\begin{align*}
\overline{Z}_{r,\gamma}(R)=R^{-\gamma}\left\|\Theta\right\|_{L^r\left(A_{R}\right)},\;\overline{Z}_{r,\gamma,\nu}(R)=R^{-\gamma}(\ln R)^{-\nu}\left\|\Theta\right\|_{L^r\left(A_{R}\right)}.
\end{align*}
We observe that all terms involving $\theta$ in \eqref{equ1.1} are derivative terms.
Thanks to this special structure of \eqref{equ1.1}, $(u,\pi,v,\Theta)$ also satisfies \eqref{equ1.1}.
By a similar argument to the proofs of Theorem \ref{main1} and Theorem \ref{main2}, we can establish the following two results.

\begin{Prop}\label{Prop5.1}
Let $(u,\pi,v,\theta)$ be a smooth solution of \eqref{equ1.1}. Suppose the parameters satisfy {\rm Assumptions \ref{a1.1}} and {\rm\ref{a1.2}}. Furthermore, assume that one of the following assumptions holds
\begin{align*}
&\mathrm{(i)}\;p\in\left[3,\frac{9}{2}\right],\;\liminf\limits_{R\rightarrow+\infty}X_{p,\alpha}(R)=0,\;\limsup\limits_{R\rightarrow+\infty}\left[Y_{q,\beta}(R)+\overline{Z}_{r,\gamma}(R)\right]<+\infty;\\
&\mathrm{(ii)}\;p\in\left(\frac{3}{2},3\right),\;\liminf\limits_{R\rightarrow+\infty}\left[X_{p,\alpha}(R)+Y_{q,\beta}(R)+\overline{Z}_{r,\gamma}(R)\right]<+\infty.
\end{align*}
Then $u=v=0$ and $\theta$ is a constant.
\end{Prop}

\begin{Prop}\label{Prop5.2}
Let $(u,\pi,v,\theta)$ be a smooth solution of \eqref{equ1.1}. Suppose the parameters  satisfy {\rm Assumptions \ref{a1.1}, \ref{a1.2}, \ref{a1.3}, \ref{a1.4}} and {\rm\ref{a1.5}}.
Furthermore, assume that one of the following assumptions holds
\begin{align*}
&\mathrm{(i)}\;p\in\left[3,\frac{9}{2}\right],\;\lim\limits_{R\rightarrow+\infty}X_{p,\alpha}(R)=0,\;\limsup\limits_{R\rightarrow+\infty}\left[Y_{q,\beta,\mu}(R)+\overline{Z}_{r,\gamma,\nu}(R)\right]<+\infty;\\
&\mathrm{(ii)}\;p\in\left(\frac{3}{2},3\right),\;\limsup\limits_{R\rightarrow+\infty}\left[X_{p,\alpha,\lambda}(R)+Y_{q,\beta,\mu}(R)+\overline{Z}_{r,\gamma,\nu}(R)\right]<+\infty.
\end{align*}
Then $u=v=0$ and $\theta$ is a constant.
\end{Prop}

\begin{proof}[{\bf Proofs of Theorem \ref{main3} and Theorem \ref{main4}}]
Theorem \ref{main3} and Theorem \ref{main4} are direct consequences of the Poincar{\'{e}}-Sobolev inequality
$$\left\|\Theta\right\|_{L^\frac{3r}{3-r}(A_R)}\leq C\|\nabla\theta\|_{L^r (A_R)},$$
Proposition \ref{Prop5.1} and Proposition \ref{Prop5.2}, respectively.
\end{proof}

\begin{Rem}
Similarly to {\rm Remark \ref{Rem3.1}}, some strict inequalities in our assumptions of {\rm Theorem \ref{main3}} can be weakened in some special endpoint cases.
\begin{itemize}
\item[(i)] When $(p,q)\in(\frac{3}{2},3)\times[2p',6]$, the inequality in {\rm Assumption \ref{a1.7}(i)}  becomes
\begin{equation*}
\alpha+2\beta<\frac{3}{p}+\frac{6}{q}-2,
\end{equation*}
which can be replaced by the equality
$$
\alpha+2\beta=\frac{3}{p}+\frac{6}{q}-2,
$$
but the price is that we need to assume in addition that
$$\liminf\limits_{R\rightarrow+\infty}X_{p,\alpha}(R)=0,\;\limsup\limits_{R\rightarrow+\infty}[Y_{q,\beta}(R)+Z'_{r,\gamma}(R)]<+\infty,\text{ or }$$
$$\limsup\limits_{R\rightarrow+\infty}[X_{p,\alpha}(R)+Z'_{r,\gamma}(R)]<+\infty,\;\liminf\limits_{R\rightarrow+\infty}Y_{q,\beta}(R)=0.$$
\item[(ii)]
When $(p,r)\in(\frac{3}{2},3)\times[\frac{6p}{5p-3},2]$, the inequality in {\rm Assumption \ref{a1.7}(ii)} becomes
\begin{equation*}
\alpha+2\gamma<\frac{3}{p}+\frac{6}{r}-4,
\end{equation*}
which can be replaced by the equality
$$
\alpha+2\gamma=\frac{3}{p}+\frac{6}{r}-4,
$$
but the price is that we need to assume in addition that
$$\liminf\limits_{R\rightarrow+\infty}X_{p,\alpha}(R)=0,\;\limsup\limits_{R\rightarrow+\infty}[Y_{q,\beta}(R)+Z'_{r,\gamma}(R)]<+\infty,\text{ or }$$
$$\limsup\limits_{R\rightarrow+\infty}[X_{p,\alpha}(R)+Y_{q,\beta}(R)]<+\infty,\;\liminf\limits_{R\rightarrow+\infty}Z'_{r,\gamma}(R)=0.$$
\item[(iii)]
When $(q,r)\in[2,6]\times[\frac{6}{5},2]$, the inequality in {\rm Assumption \ref{a1.7}(iii)} becomes
\begin{equation*}
\beta+\gamma<\frac{3}{q}+\frac{3}{r}-3,
\end{equation*}
which can be replaced by the equality
$$\beta+\gamma=\frac{3}{q}+\frac{3}{r}-3,$$
but the price is that we need to assume in addition that
\begin{itemize}
\item[(1)] when $p\in[3,\frac{9}{2}]$, we assume
$$\liminf\limits_{R\rightarrow+\infty}Y_{q,\beta}(R)=0\;\text{ or }\liminf\limits_{R\rightarrow+\infty}Z'_{r,\gamma}(R)=0.$$
\item[(2)] when $p\in(3,\frac{3}{2})$, we assume
$$\liminf\limits_{R\rightarrow+\infty}Y_{q,\beta}(R)=0,\;\limsup\limits_{R\rightarrow+\infty}[X_{p,\alpha}(R)+Z'_{r,\gamma}(R)]<+\infty,\text{ or }$$
$$\limsup\limits_{R\rightarrow+\infty}[X_{p,\alpha}(R)+Y_{q,\beta}(R)]<+\infty,\;\liminf\limits_{R\rightarrow+\infty}Z'_{r,\gamma}(R)=0.$$
\end{itemize}
\item[(iv)] When $(p,q)\in(\frac{3}{2},3)\times[2p',6]$, {\rm Assumption \ref{a1.7}(i)}
can be replaced by $\alpha=\frac{3}{p}-1$ and  $\beta=\frac{3}{q}-\frac{1}{2}$.
\item[(v)] When $(p,r)\in(\frac{3}{2},3)\times[\frac{6p}{5p-3},2]$, {\rm Assumption \ref{a1.7}(ii)}
can be replaced by $\alpha=\frac{3}{p}-1$ and $\gamma=\frac{3}{r}-\frac{3}{2}$.
\item[(vi)] When $(q,r)\in[2,6]\times[\frac{6}{5},2]$, {\rm Assumption \ref{a1.7}(iii)} can be replaced by
$(\beta,\gamma)=(0,\frac{3}{r}-\frac{3}{2})$ with $q=2$, or $(\beta,\gamma)=(\frac{3}{q}-\frac{1}{2},0)$ with $r=\frac{6}{5}$.
\end{itemize}
\end{Rem}

\begin{Rem}
Similarly to {\rm Remark \ref{Rem3.2}}, {\rm Theorem \ref{main3}} includes $17$ cases.

\end{Rem}

\subsection*{Acknowledgements.}
This work was supported by Science Foundation for the Excellent Youth Scholars of Higher Education of Anhui Province Grant No. 2023AH030073, Natural Science Foundation of the Higher Education Institutions of Anhui Province Grant No. 2023AH050478 and Domestic Study and Research Support Program for Young Key Teachers of Higher Education of Anhui Province Grant No. JNFX2025027.

\subsection*{Data Availability Statement}
No data was used for the research described in the article.

\subsection*{Conflict of Interest Statement}
The authors have no relevant financial or non-financial interests to disclose.

\end {document}